\documentclass[AMA,STIX1COL]{WileyNJD-v2}
\usepackage{moreverb}

\articletype{Article Type}%

\received{<day> <Month>, <year>}
\revised{<day> <Month>, <year>}
\accepted{<day> <Month>, <year>}

\usepackage{graphicx} 
\usepackage{epstopdf}

\usepackage{subcaption}
\usepackage{caption}
\usepackage{amsmath}

\usepackage{amssymb}
\usepackage{booktabs}
\usepackage{multirow}

\usepackage{hyperref}

\usepackage{enumitem}

\newcommand{\barr}{\begin{array}}
	\newcommand{\earr}{\end{array}}

\newcommand{\Z}{{\mathbb{Z}}}

\newcommand{\Ts}{T_{\mathrm{s}}}

\newcommand{\R}{{\mathbb{R}}}

\newcommand{\nx}{n_{\mathrm{x}}}
\newcommand{\nU}{n_{\mathrm{u}}}
\newcommand{\nc}{n_{\mathrm{c}}}

\newcommand{\z}{\boldsymbol{z}}
\renewcommand{\v}{\boldsymbol{v}}
\newcommand{\y}{\boldsymbol{y}}
\newcommand{\lam}{{\lambda}}

\newcommand{\PS}{\mathrm{PS}}

\newcommand{\diag}[1]{\text{diag}\left( {#1} \right)}

\newcommand{\fixed}{\mathcal{R}}

\newcommand{\wmax}{w_{\textrm{max}}}
\newcommand{\wmin}{w_{\textrm{min}}}
\newcommand{\inact}{{\textrm{in}}}
\newcommand{\act}{{\textrm{act}}}
\newcommand{\guess}{{\textrm{g}}}
\newcommand{\Iwinactive}{\mathcal{I}_{\textrm{in}}}
\newcommand{\Iwactive}{\mathcal{I}_{\textrm{act}}}
\newcommand{\Iwguess}{\mathcal{I}_{\textrm{g}}}

\newcommand{\Set}[2]{\{\, #1 \mid #2 \,\}}

\newcommand{\blo}{\ubar{c}}
\newcommand{\bup}{\bar{c}}

\newcommand{\px}{p^{\mathrm{x}}}
\newcommand{\py}{p^{\mathrm{y}}}
\newcommand{\vx}{v^{\mathrm{x}}}
\newcommand{\vy}{v^{\mathrm{y}}}
\newcommand{\ax}{a^{\mathrm{x}}}
\newcommand{\ay}{a^{\mathrm{y}}}
\newcommand{\goal}{g^{\mathrm{s}}}
\newcommand{\deltag}{\delta^{\mathrm{g}}}

\newcommand{\pc}{p^{\mathrm{c}}}
\newcommand{\vc}{v^{\mathrm{c}}}
\newcommand{\thetap}{\theta^{\mathrm{p}}}
\newcommand{\omegap}{\omega^{\mathrm{p}}}
\newcommand{\mc}{m^{\mathrm{c}}}
\renewcommand{\mp}{m^{\mathrm{p}}}
\newcommand{\Fc}{F^{\mathrm{in}}}
\newcommand{\Fl}{F^{\mathrm{l}}}
\newcommand{\Fr}{F^{\mathrm{r}}}

\newcommand{\pt}{p^{\mathrm{t}}}
\newcommand{\dist}{s}
\newcommand{\distl}{s^{\mathrm{l}}}
\newcommand{\distr}{s^{\mathrm{r}}}
\newcommand{\dotdistl}{\dot{s}^{\mathrm{l}}}
\newcommand{\dotdistr}{\dot{s}^{\mathrm{r}}}

\newcommand{\M}{M}
\newcommand{\xgoal}{x^{\mathrm{goal}}}
\newcommand{\pxmax}{p^{\mathrm{x},\mathrm{max}}}
\newcommand{\pxmin}{p^{\mathrm{x},\mathrm{min}}}
\newcommand{\pymax}{p^{\mathrm{y},\mathrm{max}}}
\newcommand{\pymin}{p^{\mathrm{y},\mathrm{min}}}

\newcommand{\bin}{\delta}
\newcommand{\I}{\mathcal{I}}
\renewcommand{\P}{\mathcal{P}}

\newcommand{\flag}{\texttt{flag}}

\newcommand{\True}{\texttt{True}}
\newcommand{\False}{\texttt{False}}

\theoremstyle{plain}

\newtheorem{Lemma}[theorem]{Lemma}

\theoremstyle{Definition}
\newtheorem{Definition}[theorem]{Definition}
\newtheorem{Example}[theorem]{Example}

\theoremstyle{Remark}
\newtheorem{Remark}{Remark}

\renewcommand{\OR}{$\Vert$}

\usepackage{accents}
\newcommand{\ubar}[1]{\underaccent{\bar}{#1}}

\usepackage{color}

\usepackage{algorithm}
\usepackage{algpseudocode}
\algrenewcommand\algorithmicrequire{\textbf{Input:}}
\algrenewcommand\algorithmicensure{\textbf{Output:}}

\algnewcommand{\IfThenElse}[3]{
	\State \algorithmicif\ #1\ \algorithmicthen\ #2\ \algorithmicelse\ #3}
\algnewcommand{\IfThen}[2]{
\State \algorithmicif\ #1\ \algorithmicthen\ #2}

\newcommand{\software}[1]{\texttt{#1}}

\usepackage{color}
\definecolor{wheat}{rgb}{0.96,0.87,0.70}

\newcommand{\change}[1]{#1}

\begin{document}
	
	\title{Tailored Presolve Techniques in Branch-and-Bound Method for Fast Mixed-Integer Optimal Control Applications
	\protect\thanks{Tailored Presolve Techniques in Branch-and-Bound Method for Fast Mixed-Integer Optimal Control Applications}}

	\author[1]{Rien Quirynen*}
	
	\author[1]{Stefano Di Cairano}
	
	\authormark{R. QUIRYNEN \textsc{et al}}

	\address[1]{\orgdiv{Control for Autonomy}, \orgname{Mitsubishi Electric Research Laboratories}, \orgaddress{\state{Massachusetts}, \country{USA}}}
	
	
	\corres{*\email{quirynen@merl.com}}
	
	\presentaddress{201 Broadway, 8th Floor, Cambridge, MA 02139-1955}
	
	\abstract[Abstract]{
		Mixed-integer model predictive control~(MI-MPC) can be a powerful tool for controlling hybrid systems. 
		In case of a linear-quadratic objective in combination with linear or piecewise-linear system dynamics and inequality constraints, MI-MPC needs to solve a mixed-integer quadratic program~(MIQP) at each sampling time step. 
		This paper presents a collection of exact block-sparse presolve techniques to efficiently remove decision variables, and to remove or tighten inequality constraints, tailored to mixed-integer optimal control problems. In addition, we describe a novel approach based on a heuristic presolve algorithm to compute a feasible but possibly suboptimal MIQP solution. We present benchmarking results for a C~code implementation of the proposed \software{BB-ASIPM} solver, including a branch-and-bound~(B\&B) method with the proposed tailored presolve techniques and an active-set based interior point method~(ASIPM), compared against multiple state-of-the-art MIQP solvers on a case study of motion planning with obstacle avoidance constraints. Finally, we demonstrate the feasibility and computational performance of the \software{BB-ASIPM} solver in embedded system on a dSPACE Scalexio real-time rapid prototyping unit for a second case study of stabilization for an underactuated cart-pole with soft contacts.
	}
	
	
	\jnlcitation{\cname{%
			\author{Quirynen R.}, and
			\author{Di Cairano S.}} (\cyear{2022}), 
		\ctitle{Tailored Presolve Techniques in Branch-and-Bound Method for Fast Mixed-Integer Optimal Control Applications}, \cjournal{Optimal Control Applications and Methods}, \cvol{2022;00:1--6}.}
	
	\maketitle
	
		

\section{Introduction}

Optimization-based motion planning and control techniques, such as model predictive control~(MPC), allow a model-based design framework in which the dynamics, constraints and objectives are directly taken into account~\cite{Mayne2013}. This framework has been extended to hybrid systems~\cite{bemporad1999control}, providing a powerful technique to model a large range of hybrid control problems, e.g., including switched dynamical systems~\cite{MarcucciTedrake2019,ChenCulbertsonEtAl2021}, discrete/quantized actuation~\cite{WalshWeissDiCairano16}, motion planning with obstacle avoidance~\cite{LandryDeitsEtAl2016}, logic rules and temporal logic specifications~\cite{SahinQuirynenEtAl2020}.
However, the resulting optimal control problems~(OCPs) are non-convex combinatorial optimization problems, because they contain variables that only take integer or binary values, and they are $\mathcal{NP}$-hard to solve in general~\cite{Pia2017}. When using a linear-quadratic objective in combination with linear or piecewise-linear system dynamics and inequality constraints, the resulting OCP can be formulated as a mixed-integer quadratic program~(MIQP).

In the present work, we aim to solve MIQP problems of the following form:
\begin{subequations} \label{eq:OCP-MIQP}
\begin{alignat}{5}
\hspace{-3mm}\underset{X,\,U}{\text{min}} \quad &\sum_{i=0}^{N} \frac{1}{2}\begin{bmatrix}	x_i \\ u_i	\end{bmatrix}^\top H_i \begin{bmatrix}	x_i \\ u_i	\end{bmatrix} + \begin{bmatrix}	x_i \\ u_i	\end{bmatrix}^\top \begin{bmatrix}	q_i \\ r_i	\end{bmatrix} && \label{OCP:obj}\\
\hspace{-3mm}\text{s.t.} \quad\; 
& x_{i+1} = \begin{bmatrix} A_i & B_i \end{bmatrix} \begin{bmatrix} x_i \\ u_i \end{bmatrix} + a_{i}, && \hspace{-2em} \forall i \in \{0,\ldots,N-1\}, \label{OCP:dyn}\\
& \begin{bmatrix} \ubar{x}_i \\ \ubar{u}_i \end{bmatrix} \leq \begin{bmatrix} x_i \\ u_i \end{bmatrix} \leq \begin{bmatrix} \bar{x}_i \\ \bar{u}_i \end{bmatrix}, \quad &&\hspace{-2em} \forall i \in \{0,\ldots,N\}, \label{OCP:bounds} \\
& \blo_{i} \leq \begin{bmatrix} C_i & D_i \end{bmatrix} \begin{bmatrix} x_i \\ u_i \end{bmatrix} \leq \bup_{i}, \quad &&\hspace{-2em} \forall i \in \{0,\ldots,N\}, \label{OCP:ineq} \\
& u_{i,j} \in \Z, \quad \forall j \in \I_i, && \hspace{-2em} \forall i \in \{0,\ldots,N\}, \label{OCP:int}
\end{alignat}
\end{subequations} 
where the state variables are $x_i \in \R^{\nx^i}$, the control variables are $u_i \in \R^{\nU^i}$ and $\I_i$ denotes the index set of integer decision variables, i.e., the cardinality $| \I_i | \le \nU^i$ denotes the number of integer variables at each time step $i \in \{0,1,\ldots,N\}$.
The objective in~\eqref{OCP:obj} defines a linear-quadratic function with positive semi-definite Hessian matrix $H_i \succeq 0$ and gradient vectors $q_i \in \R^{\nx^i}$ and $r_i \in \R^{\nU^i}$. The constraints include state dynamic equality constraints in~\eqref{OCP:dyn}, simple bounds in~\eqref{OCP:bounds}, affine inequality constraints in~\eqref{OCP:ineq} and integer feasibility constraints in~\eqref{OCP:int}. \change{The optimization problem in~\eqref{eq:OCP-MIQP} is an optimal control structured MIQP or a block-sparse MIQP due to the block-structured sparsity of the Hessian and constraint matrix, i.e., the objective in~\eqref{OCP:obj} and inequality constraints in~\eqref{OCP:bounds}-\eqref{OCP:ineq} are separable per time step $i \in \{0,1,\ldots,N\}$ but the variables are coupled via the dynamics in~\eqref{OCP:dyn}.} Note that an initial state constraint $x_0 = \hat{x}_0$, where $\hat{x}_0$ is a current state estimate, can be enforced using the simple bounds in~\eqref{OCP:bounds}.
A compact notation is used to denote the optimization variables of the MIQP in~\eqref{eq:OCP-MIQP} as the state $X=[x_0^\top,\ldots, x_N^\top]^\top$ and control trajectory $U=[u_0^\top,\ldots, u_{N}^\top]^\top$. In the present paper, we occasionally refer to all optimization variables in~\eqref{eq:OCP-MIQP} as $Z=[z_0^\top,\ldots, z_N^\top]^\top$, where $z_i = [x_i^\top, u_i^\top]^\top$ for $i \in \{0,1,\ldots,N\}$.
Unlike standard OCP formulations, the \change{MIQP} in~\eqref{eq:OCP-MIQP} \change{defines a mixed-integer OCP~(MIOCP) that} includes control variables on the terminal stage, $u_N \in \R^{\nU^N}$, which may include auxiliary variables to formulate the mixed-integer inequality constraints of the hybrid control system. A binary optimization variable $u_{i,j} \in \{0,1\}$ can be defined as an integer variable $u_{i,j} \in \Z$ in~\eqref{OCP:int}, including the simple bounds $0 \le u_{i,j} \le 1$ in~\eqref{OCP:bounds}. Without loss of generality and for simplicity of notation, the integer optimization variables in~\eqref{eq:OCP-MIQP} are restricted to control variables, even though the methods in this paper can be trivially extended to MIOCPs with integer state variables. 
MPC for any hybrid system can be formulated as in~\eqref{eq:OCP-MIQP}, for example, by leveraging a mixed logical dynamical~(MLD) model~\cite{bemporad1999control}. There are typically different ways to formulate a hybrid control problem in the form of the MIQP in~\eqref{eq:OCP-MIQP}, providing a tradeoff between the number of decision variables and strength of the formulation~\cite{nemhauser1988integer}. The strength of an MIQP corresponds to how close the convex relaxations are to the exact MIQP solution, e.g., see the formulations in~\cite{MarcucciTedrake2019} for piecewise-affine systems.

Mixed-integer MPC~(MI-MPC) implementations for motion planning and control aim to solve the MIOCP in~\eqref{eq:OCP-MIQP} at each sampling time step. This is challenging due to the $\mathcal{NP}$-hard complexity of solving MIQPs in general~\cite{Pia2017}, and given the relatively small computational resources and available memory on embedded microprocessors for real-time control applications~\cite{DiCairano2018tutorial}. Therefore, several tailored solution strategies have been proposed for MI-MPC. These approaches can \change{generally} be divided into \emph{heuristic} techniques, which seek to efficiently find feasible but suboptimal solutions to the problem, and \emph{exact} algorithms that solve the MIQPs to optimality. Examples of the former include rounding schemes~\cite{Kirches2010y,Sager2008}, the feasibility pump~\cite{achterberg2007improving}, approximate optimization~\cite{diamond2016general,naik2017embedded}, approximate dynamic programming~\cite{Stellato2017}, and approximate explicit hybrid MPC~\cite{MarcucciDeitsEtAl2017}. Machine learning can be used to approximately solve combinatorial optimization problems~\cite{BengioLodiEtAl2018}, e.g., using supervised learning to train a network architecture to quickly compute feasible but suboptimal solutions online as in~\cite{KargLucia2018,MastiBemporad2019,BertsimasStellato2019,LoehrKlaucoEtAl2020,Cauligi2022a,Cauligi2022}. 
The downside of fast heuristic approaches is often the lack of guarantees of finding an optimal, or even a feasible, solution.

Due to the complexity of solving MIQPs, explicit methods that compute \emph{offline} the optimal control as a function of the system parameters have been developed, e.g., see~\cite{Borrelli2005,Oberdieck2015}. The application of these explicit methods is generally limited to small-dimensional systems with few discrete variables. We therefore focus on the use of \emph{online} solution methods in the present paper. The MIQP in~\eqref{eq:OCP-MIQP} is a mixed-integer convex program~(MICP), i.e., it becomes a convex problem after relaxing the integer feasibility constraints in~\eqref{OCP:int}. Most of the exact optimization algorithms for MIQPs are based on the classical branch-and-bound~(B\&B) technique~\cite{floudas1995nonlinear}. Specifically for the structured \change{MIQP} in~\eqref{eq:OCP-MIQP}, the B\&B strategy has been combined with tailored algorithms for solving the relaxed convex QPs. For example, a B\&B algorithm for MI-MPC has been proposed in combination with a dual active-set solver in~\cite{axehill2006mixed}, a primal active-set solver in~\cite{Hespanhol2019,PRESAS}, an interior point algorithm in~\cite{frick2015embedded}, dual projected gradient methods in~\cite{Axehill2008,naik2017embedded}, a nonnegative least squares solver in~\cite{Bemporad2018}, and with the alternating direction method of multipliers~(ADMM) in~\cite{stellato2018embedded}. Machine learning could also be used to speed up the exact solution of combinatorial optimization problems, e.g., by improving the tree search in B\&B methods~\cite{BengioLodiEtAl2018,balcan2018learning}. B\&B methods for solving mixed-integer nonlinear OCPs have also been studied, e.g., in~\cite{Gerdts2005}.

The computational efficiency of B\&B methods is affected by the node selection and branching rules~\cite{achterberg2005branching,le2017abstract}, and by the convex relaxation solutions~\cite{luo2010semidefinite,Axehill2010}, but there are many other important factors that make state-of-the-art MIQP solvers like \software{GUROBI}~\cite{gurobi} and \software{MOSEK}~\cite{mosek} successful. A crucial algorithmic ingredient is the \emph{presolve} routine~\cite{AchterbergBixbyEtAl2019}, which typically is called in each node of the B\&B method (see Figure~\ref{fig:BB_ilu}) before solving the convex relaxation, and it
performs a collection of operations to remove decision variables, and to remove or tighten constraints. 
In the present paper, we explicitly refer to these presolve techniques as \emph{exact} operations to emphasize that they preserve feasibility and optimality, i.e., a feasible and optimal solution to the reduced problem exists as long as a feasible and optimal solution exists to the original MIQP.
Exact presolve techniques are vital for the good performance of current state-of-the-art MICP solvers, such that B\&B methods can often solve seemingly intractable problems in practice~\cite{Achterberg2013}.
Especially for MI-MPC applications, warm starting strategies exist that aim to reuse the explored B\&B tree at one time step to reduce the computational cost of the B\&B method at the next control time step. In recent years, different variants of B\&B warm starting have been proposed for MI-MPC, e.g., in~\cite{Bemporad2018,Hespanhol2019,Marcucci2021}.

Unlike state-of-the-art mixed-integer solvers, e.g., \software{GUROBI}~\cite{gurobi} and \software{MOSEK}~\cite{mosek}, our aim is to propose a tailored algorithm and its solver implementation for fast embedded MI-MPC applications, i.e., running on microprocessors with considerably less computational resources and available memory~\cite{DiCairano2018tutorial}, while leveraging the special structure of the \change{MIQP} in~\eqref{eq:OCP-MIQP}. The optimization algorithm should be relatively simple to code with a moderate use of resources, while the software implementation is preferably compact and library independent. In the present paper, we will use a tailored active-set based interior point method~(ASIPM) that was presented in~\cite{ASIPM} to solve the block-sparse convex QP relaxations in the B\&B method, resulting in an MIQP solver that will further be referred to as \software{BB-ASIPM}. The recent work in~\cite{Liang2021} showed how infeasibility detection and early termination based on duality can be implemented efficiently for an infeasible primal-dual interior point method~(IPM) based on a computationally efficient projection strategy that will be used within our \software{BB-ASIPM} solver.

{\emph{Our Contributions:} A first contribution of the present paper is an exact block-sparse presolve routine that is tailored to \change{MIQPs} of the form in~\eqref{eq:OCP-MIQP}.
Our previous work in~\cite{Hespanhol2019,Liang2021} showed how domain propagation~\cite{AchterbergBixbyEtAl2019,Savelsbergh1994} can be applied to the condensed form of~\eqref{eq:OCP-MIQP}. In the present paper, we propose an algorithm for domain propagation that can be applied directly to the \change{MIQP} in~\eqref{eq:OCP-MIQP}, based on a forward-backward propagation of variable bounds, in order to speed up the computation times of a B\&B method. We additionally present tailored algorithms for other presolve techniques, including the removal of trivial constraints, dual fixings, constraint coefficient strengthening and binary variable probing. A second contribution involves the use of the proposed presolve routine in a heuristic procedure to compute a feasible but possibly suboptimal MIQP solution. \change{This heuristic} is an extension of the idea in~\cite{Cauligi2022} for improving supervised learning \change{of MICP solutions}. A third contribution is the benchmarking results for a C~code implementation of the \software{BB-ASIPM} algorithm against state-of-the-art MIQP solvers, including \software{GUROBI}~\cite{gurobi}, \software{MOSEK}~\cite{mosek}, \software{GLPK}~\cite{GLPK2022}, \software{Cbc}~\cite{CBC2022}, and Matlab's \software{intlinprog}, based on a case study of mobile robot motion planning with obstacle avoidance constraints. A fourth and final contribution is the demonstration of the computational performance of the \software{BB-ASIPM} solver on a dSPACE Scalexio rapid prototyping unit, using a second case study of stabilization for an underactuated cart-pole with soft contacts.

The paper is organized based on each of the following algorithmic ingredients that are typically important for a good computational performance of a B\&B method for mixed-integer optimal control applications:
\begin{itemize}[noitemsep,topsep=0pt]
		\item variable branching decisions and node selection strategies~(Section~\ref{sec:MIQP}),
		\item structure-exploiting convex solver for efficient QP solutions~(Section~\ref{sec:ASIPM}),
		\item exact presolve techniques for variable fixings and tight relaxations~(Section~\ref{sec:PRESOLVE}),
		\item fast primal heuristic algorithm to find integer-feasible solutions~(Section~\ref{sec:HEURISTICS}),
		\item warm starting and embedded software for mixed-integer MPC~(Section~\ref{sec:MIMPC}).
\end{itemize}
The performance of our proposed MIQP solver is illustrated based on two MI-MPC case studies in Section~\ref{sec:caseStudies}, including hardware-in-the-loop simulations on a dSPACE Scalexio rapid prototyping unit. Finally, Section~\ref{sec:concl} concludes the paper.

\section{Preliminaries on Mixed-Integer Quadratic Programming} \label{sec:MIQP}

We first introduce some of the basic concepts in mixed-integer quadratic programming~(MIQP) solvers based on branch-and-bound~(B\&B) methods, such as convex QP relaxations, node selection and branching strategies.

\subsection{Branch-and-Bound Algorithm}
\label{sec:BB_OCP}

\begin{figure} 
	\centering
	\includegraphics[trim={12cm 0 0 0},width=0.46\textwidth]{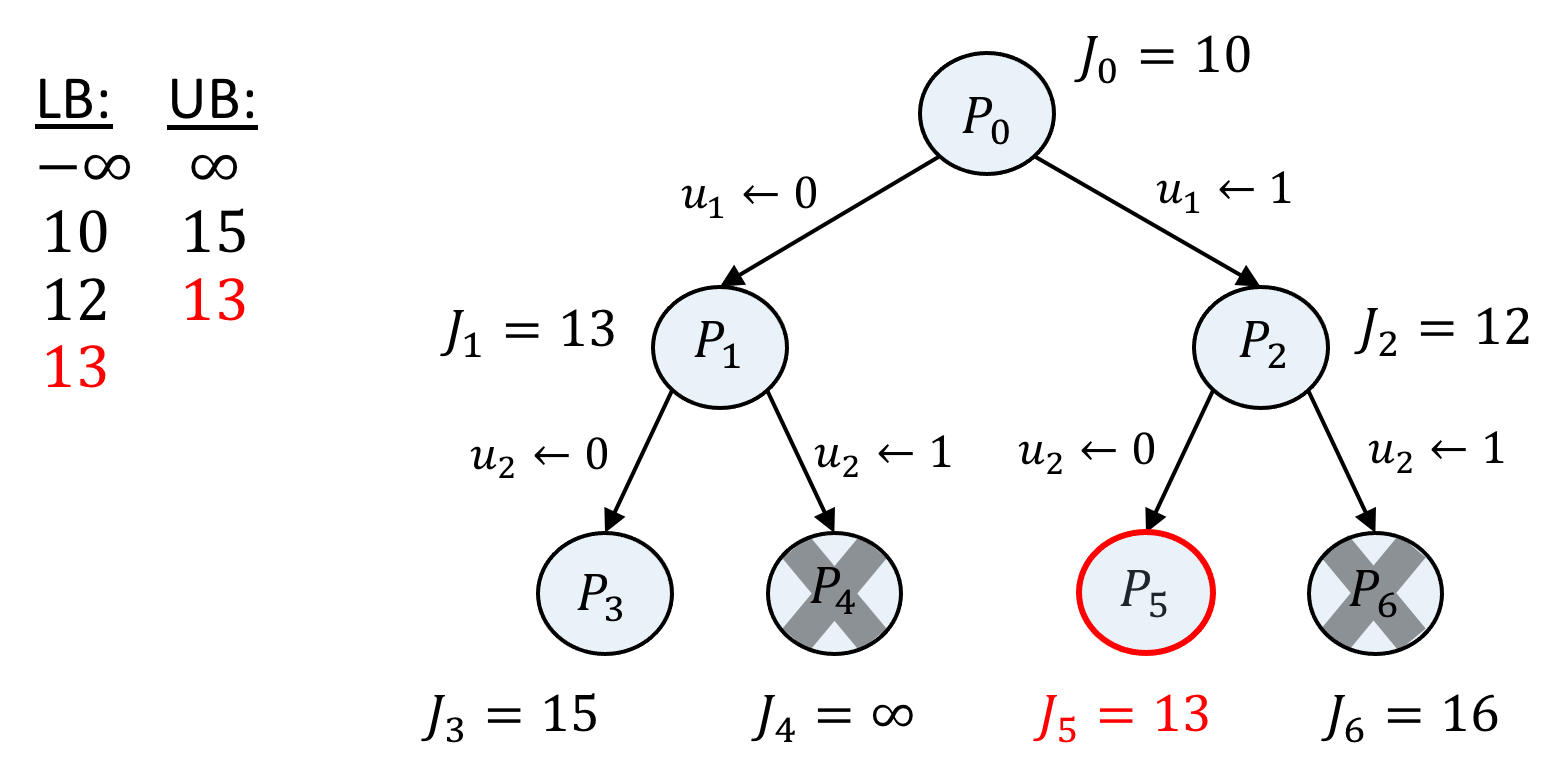}
	\caption{Illustration of the branch-and-bound~(B\&B) method as a binary search tree. A selected node can be either \emph{branched}, e.g., resulting in two partitions for each binary variable $u_{i,j} \in \{0, 1\}$, or \emph{pruned} based on feasibility or the current upper bound.}
	\label{fig:BB_ilu}
\end{figure}

The main idea of the B\&B optimization algorithm is to sequentially create partitions of the original MIQP problem and attempt to solve each of these partitions. While solving each partition may still be challenging, it is fairly efficient to obtain local lower bounds on the optimal objective value, e.g., by solving convex relaxations of the MIQP subproblem. If we happen to obtain an integer-feasible solution while solving a relaxation, we can then use it to obtain a global upper bound for the solution to the original problem. This may help to avoid solving or branching certain partitions that were already created, i.e., such partitions or nodes can be \emph{pruned}. The general algorithmic idea of partitioning is better illustrated as a binary search tree, see Figure~\ref{fig:BB_ilu}.
A key step in this approach is how to create the partitions, i.e., which node to choose and which discrete variable to select for branching. Since we solve a convex relaxation at every node of the tree, it is natural to branch on one of the discrete variables with fractional values in the optimal solution of the relaxation. Therefore, if a binary variable, e.g., $u_{i,j} \in \{0, 1\}$ has a fractional value in a convex relaxation, we create two partitions where we add the equality constraints $u_{i,j} = 0$ and $u_{i,j} = 1$, respectively.
Another key choice is the order in which the created subproblems are solved. These steps have been extensively explored in the literature and various heuristics are implemented in state-of-the-art tools~\cite{achterberg2005branching}.

\subsection{Convex Quadratic Program Relaxations}
\label{sec:qprel}

We use a B\&B method to solve the MIQP~\eqref{eq:OCP-MIQP} by solving convex quadratic programming~(QP) relaxations that are constructed by dropping the integer feasibility constraints in~\eqref{OCP:int}, e.g., $u_{i,j} \in \{0,1\}$ is relaxed to a bounded continuous variable $0 \le u_{i,j} \le 1$. Other convex relaxations for MIQPs have been studied in the literature such as moment or SDP relaxations that may be tighter than QP relaxations~\cite{luo2010semidefinite,Axehill2010}, however they are often relatively expensive to solve, even if they may drastically reduce the nodes explored in a B\&B method. In this paper, we restrict to standard QP relaxations and we present a tailored active-set based interior point method~(ASIPM) in Section~\ref{sec:ASIPM}. The \change{\software{ASIPM} solver} has been shown to be competitive with state-of-the-art QP solvers for embedded MPC~\cite{ASIPM}, it benefits from warm-starting 
\change{and it allows for an efficient implementation of infeasibility detection and early termination based on duality~\cite{Liang2021}}. For using the \software{ASIPM} solver, the relaxations need to be convex, i.e., the Hessian matrices $H_i$ need to be positive semi-definite in~\eqref{OCP:obj} such that each solution to a QP relaxation is globally optimal.

\subsection{Tree Search: Node Selection Strategies}

A common implementation of the B\&B method is based on a \emph{depth-first} node selection strategy, which can be implemented by a last-in-first-out~(LIFO) buffer. The next node to be solved is selected as one of the children of the current node and this process is repeated until integer feasibility or until a node is pruned because the node is either infeasible or dominated by the upper bound, which is followed by a backtracking procedure. Instead, a \emph{best-first} strategy selects the node with the lowest local lower bound so far. In this paper, we will employ a combination of the depth-first and best-first node selection approach. This \change{combination} aims to find an integer-feasible solution quickly at the start of the B\&B procedure~(depth-first) to allow for early pruning, followed by a more greedy search for better feasible solutions~(best-first).

\subsection{Reliability Branching for Variable Selection}

Many branching rules exist such as ``most infeasible'' branching which selects the integer variable with fractional part in the QP relaxation that is closest to $0.5$. Even though \change{this rule} is used quite often, e.g., in~\cite{Bemporad2018,stellato2018embedded}, it generally does not perform very well in practice~\cite{achterberg2005branching}. We instead use \emph{reliability branching} which is based on a combination of two concepts for variable selection: strong branching and pseudo-costs~\cite{achterberg2005branching}. Strong branching relies on temporarily branching, both up (to higher integer) and down (to lower integer), for every integer variable that has a fractional value in the solution of a convex QP relaxation in a given node, before committing to the variable that provides the highest value for a particular score function. The increase in objective values $\Delta_{i,j}^+$, $\Delta_{i,j}^-$ are computed when branching the integer variable $u_{i,j}$, respectively, up and down. Given these quantities, a simple scoring function $\text{score}(\cdot,\cdot)$ is computed for each integer variable, such as the product score function~\cite{le2017abstract}
\begin{equation}
S_{i,j} = \text{score}(\Delta_{i,j}^-,\Delta_{i,j}^+) = \text{max}(\Delta_{i,j}^+,\epsilon) \, \cdot \, \text{max}(\Delta_{i,j}^-,\epsilon), \label{eq:scorefun}
\end{equation}
given a small positive value $\epsilon > 0$.
Full strong branching has been empirically shown to provide smaller search trees in practice~\cite{achterberg2005branching}, but it is relatively expensive since several QP relaxations are solved in order to select one variable to branch on.

%

The idea of pseudo-costs aims at approximating the increase of the objective function to decide which variable to branch on, without solving additional QP relaxations. This can be done by keeping statistic information for each integer variable, i.e., the \emph{pseudo-costs} that represent the average increase in the objective value per unit change in that particular integer variable when branching. The current pseudo-cost values are computed based on branching decisions that occurred in different parts of the B\&B tree~\cite{achterberg2005branching}.
Each variable has two pseudo-costs, $\phi_{i,j}^{-}$ when the variable is branched ``down'' and $\phi_{i,j}^{+}$ when it is branched ``up''. 
However, at the beginning of the B\&B algorithm, the pseudo-costs are not yet initialized, which is when branching decisions typically impact the tree size the most. \emph{Reliability branching} uses strong branching to initialize the pseudo-costs until a certain condition of reliability is satisfied, e.g., it moves to using pseudo-costs for a particular variable once it has been branched on a specified number $\eta_{rel}$ of times~\cite{achterberg2005branching}. 
Thus, reliability branching coincides with pseudo-cost branching if $\eta_{rel}=0$, with strong branching if $\eta_{rel}=\infty$, but typically a value $1 \le \eta_{rel} \le 3$ is chosen.
This rule is further augmented by implementing a look ahead limit in the number of candidates, as well as a limit on the number of QP iterations in the strong branching step.

\section{Convex Relaxation Solver: Active-set Interior Point Method~(ASIPM)} \label{sec:ASIPM}

We reformulate the MIQP in~\eqref{eq:OCP-MIQP} by the following compact notation
\begin{subequations} \label{eq:MIQP}
	\begin{alignat}{5}
	\underset{\z}{\text{min}} \quad & \frac{1}{2} {\z}^{\top} H\, \z + h^\top \z \label{eq:MIQP-primal} \\
	\text{s.t.} \quad\; & G\, \z \;\le \;g, \quad F\, \z \;&&= \;f, \label{eq:MIQP-GF} \\
	\quad\; & \z_{j} \in \Z, \quad && j \in \I, \label{MIQP:int}
	\end{alignat}
\end{subequations} 
where $\z$ includes all primal optimization variables and the index set $\I$ denotes the integer variables. For many practical MPC formulations, the objective can be written as
$\frac{1}{2} {\z}^{\top} H\, \z + h^\top \z = \frac{1}{2}\v^{\top}Q \v + h_{\v}^\top \v + h_{\y}^\top \y$,
where $H \succeq 0$ and $Q \succ 0$, by partitioning $\z$ into $\v$ and $\y$, entering in the linear-quadratic and linear-only terms, respectively. The objective of any MIQP~\eqref{eq:MIQP} can be reformulated in the latter form by a change of variables, e.g., based on the eigenvalue decomposition for the Hessian $H \succeq 0$.



We focus on the efficient solution of convex QP relaxations in a B\&B optimization method. For a particular node of the B\&B tree, a convex QP is obtained by relaxing the integer feasibility constraint in~\eqref{MIQP:int} as
$\ubar{\z}_{j} \le \z_{j} \le \bar{\z}_{j}, \; \forall j \in \tilde{\I}$,
where the index set $\tilde{\I} \subseteq \I$ denotes each integer variable that has not been fixed due to branching, or due to the presolve routine, in the current node of the B\&B tree. The values $\ubar{\z}_{j}$ and $\bar{\z}_{j}$ denote the lower and upper bound values for each integer variable $j \in \tilde{\I}$, respectively. 
This section describes an overview of recent work on the efficient implementation of infeasibility detection and early termination based on duality~\cite{Liang2021}, applied to the active-set based interior point method~(ASIPM) that was proposed in~\cite{ASIPM}.

\subsection{Dual QP Problem Formulation}
We consider the primal convex QP of the form
\begin{subequations} \label{eq:QP}
	\begin{alignat}{5}
	\underset{\v, \y}{\text{min}} \quad & \phi(\v,\y) := \frac{1}{2}\v^{\top}Q \v + h_{\v}^\top \v + h_{\y}^\top \y \label{eq:QP-primal} \\
	\text{s.t.} \quad\; & G_{\v}\v + G_{\y} \y \;\le \;g, \label{eq:QP-G} \\
	\quad\; & F_{\v} \v + F_{\y} \y \;= \;f, \label{eq:QP-F}
	\end{alignat}
\end{subequations} 
where $ Q \succ 0 $ in the primal objective $\phi(\v, \y)$, and the inequality constraints~\eqref{eq:QP-G} include both the original inequalities from~\eqref{eq:MIQP-GF} and the convex relaxations of the integer feasibility constraints. We additionally define the compact notation $ \z:=[\v^\top \y^\top]^\top $, $ G:=[G_{\v} |G_{\y}] $, $ F:=[F_{\v}| F_{\y}] $, $ h:=[h_{\v}^\top h_{\y}^\top]^\top $ and $ H:=\begin{bmatrix} Q & 0\\ 0 & 0 \end{bmatrix} $.
The dual QP of~\eqref{eq:QP} reads as
\begin{subequations} \label{eq:dual-QP}
	\begin{alignat}{5}
	\max_{\mu,\lam} \quad & \psi(\mu,\lam) := -\frac{1}{2}\|\hat h(\mu,\lam)\|^2_{Q^{-1}} 
	-\begin{bmatrix} g\\ f \end{bmatrix}^\top \begin{bmatrix} \mu\\ \lam \end{bmatrix}
	\label{eq:dual}\\
	\text{s.t.} 
	\quad\; & G_{\y}^\top \mu + F_{\y}^\top \lam \;= \;-h_{\y}, \label{eq:dual-eq} \\
	\quad\; & \mu \;\ge \;0, \label{ineq:mu}
	\end{alignat}
\end{subequations} 
where $\lam$ and $\mu$ denote the Lagrange multipliers for the equality and inequality constraints, respectively, and 
$\hat h(\mu,\lam) := h_{\v} + G_{\v}^\top \mu + F_{\v}^\top \lam$
is defined for simplifying the dual objective function $\psi(\mu,\lam)$ in~\eqref{eq:dual}.

\subsection{Primal-dual Interior Point Method}\label{subsec:IPM}

A primal-dual IPM uses a Newton-type method to solve a sequence of relaxed Karush-Kuhn-Tucker~(KKT) conditions for the convex QP in~\eqref{eq:QP}. An iteration of the IPM typically solves the reduced linear system~\cite{Wright1997}
\begin{equation}\label{eq:IPM}
\begin{bmatrix} H & F^\top & G^\top \\F & 0 & 0 \\G & 0 & -W^k \end{bmatrix}
\begin{bmatrix} \Delta \z^k\\\Delta \lam^k\\\Delta \mu^k \end{bmatrix}
= -\begin{bmatrix} r_{\z}^k\\r_\lam^k\\\bar{r}_\mu^k \end{bmatrix},
\end{equation}
where $W^k = \diag{w^k} \succ 0$, $w^k \in \R^{n_{\mathrm{ieq}}}$ and $w_i^k = s_i^k/\mu_i^k > 0$, given current values of the slack variables $s^k \in \R^{n_{\mathrm{ieq}}}$ and Lagrange multipliers $\mu^k \in \R^{n_{\mathrm{ieq}}}$ for the inequality constraints in the $k^{\text{th}}$ iteration of the Newton-type method.
The right-hand side in~\eqref{eq:IPM} denotes the residual value for the optimality conditions and reads as
\begin{equation}\label{def:rhs}
\begin{aligned}
r_{\z}^k &= H \z^k + F^\top \lam^k + G^\top \mu^k + h, \quad r_\lam^k = F \z^k - f, \\
r_\mu^k &= G \z^k - g + s^k,  \quad r_s^k = M^k S^k \mathrm{1} - \tau^k \mathrm{1}, \quad \bar{r}_\mu^k = r_\mu^k - M^{k^{-1}} r_s^k,
\end{aligned}
\end{equation}
based on the barrier parameter $\tau^k \rightarrow 0$ for $k \rightarrow \infty$, where $M^k = \diag{\mu^k}$, $S^k = \diag{s^k}$ and the slack variables are updated as $\Delta s^k = - (W^k \Delta \mu^k + M^{k^{-1}} r_s^k)$. We consider an infeasible primal-dual IPM for which the starting point $\left\{\left(\z^{0}, \mu^{0}, \lambda^{0}, s^{0}\right)\right\}$ may not be primal and/or dual feasible, but the slack variables and Lagrange multipliers are positive at each iteration, i.e., $s^k \ge 0$ and $\mu^k \ge 0$. See~\cite{Wright1997} for details on properties and implementation of primal-dual IPMs.

\begin{Remark}
	Fixed variables, i.e., $\ubar{\z}_{j} = \bar{\z}_{j}$, and redundant constraints should be removed from each convex QP for computational efficiency. In some cases, $\nx^0 = 0$ if an initial state value $x_0 = \hat{x}_0$ is imposed in~\eqref{eq:OCP-MIQP}, and many other variables are fixed due to exact presolve operations. Therefore, we require the QP solver to allow defining a varying number of state $x_i \in \R^{\nx^i}$ and control variables $u_i \in \R^{\nU^i}$, and a varying number of inequality constraints $\nc^i$ at each time step $i = 0,1,\ldots,N$ in the prediction horizon.
\end{Remark}

\subsection{Active-set based Inexact Newton Method}\label{subsec:activeSet}

In this paper, we use the active-set based inexact Newton implementation of \software{ASIPM} from~\cite{ASIPM}, which allows for block-sparse structure exploitation, reduced computations, warm starting and improved numerical conditioning.
For inequality constraints that are strictly active at the optimum, $s_i^k \rightarrow 0$ and $\mu_i^k > 0$ such that $w_i^k \rightarrow 0$ for $k \rightarrow \infty$. On the other hand, for inactive inequality constraints, $\mu_i^k \rightarrow 0$ and $s_i^k > 0$ such that $w_i^k \rightarrow \infty$ for $k \rightarrow \infty$. Thus, the $w$-values become increasingly small and large for active and inactive inequality constraints, respectively, which highlights the well-known numerical ill-conditioning that must be tackled when implementing IPMs~\cite{Shahzad2010b}.
	Based on lower and upper bound values $0 < \wmin \ll \wmax$, at each IPM iteration, we classify the inequality constraints into the following three categories:
	\begin{itemize}[noitemsep,topsep=0pt]
		\item \emph{inactive}: constraints that are likely to be inactive at the solution, with index set $\Iwinactive := \Set{i}{w_i \geq \wmax}$,
		\item \emph{active}: constraints that are likely to be active at the solution, with index set $\Iwactive := \Set{i}{w_i \leq \wmin}$,
		\item \emph{guessing}: constraints that are uncertain, i.e., not in previous categories, with index set $\Iwguess := \Set{i}{\wmin < w_i < \wmax}$.
	\end{itemize}

In the inexact Newton-type algorithm of \software{ASIPM}~\cite{ASIPM}, we solve the linearized KKT system in Eq.~\eqref{eq:IPM} approximately by solving the following reduced block-tridiagonal linear system
\begin{align}
	\left[H + \frac{1}{\varepsilon}F^\top F + \frac{1}{\wmin}G_\act^\top G_\act + G_\guess^\top W_\guess^{k^{-1}} G_\guess + \frac{1}{\wmax} G_\inact^\top G_\inact\right] \Delta z^k = -\bar{r}_z^k,
	\label{eq:ls_primal}
\end{align}
where the right-hand side reads $\bar{r}_z^k = r_z^k + \frac{1}{\varepsilon} F^\top r_{\lambda}^k + \frac{1}{\wmin} G_\act^\top \bar{r}_{\mu,\act}^k + G_\guess^\top W_\guess^{k^{-1}} \bar{r}_{\mu,\guess}^k + G_\inact^\top W_\inact^{k^{-1}} \bar{r}_{\mu,\inact}^k$, by using a block-tridiagonal Cholesky factorization for the augmented Hessian matrix in~\eqref{eq:ls_primal}. 
The inequality constraints and $w$-values have been reordered and grouped together according to their category, i.e., $W_\inact^k := W_{i \in \Iwinactive}^k$, $W_\act^k := W_{i \in \Iwactive}^k$ and $W_\guess^k := W_{i \in \Iwguess}^k$. Similarly, we split $G$ and $\bar{r}_{\mu}$ into the corresponding blocks.
The search directions for the Lagrange multipliers are computed as
\begin{subequations}
		\begin{alignat}{5}
		\hspace{-2mm}\Delta \lambda^k &= \frac{1}{\varepsilon} \left( r_\lambda^k + F \Delta z^k \right),
		&\Delta \mu_\act^k &= \frac{1}{\wmin} \left( \bar{r}_{\mu,\act}^k + G_\act \Delta z^k \right)\!, \\
		\hspace{-2mm}\Delta \mu_\guess^k &= W_\guess^{k^{-1}} \! \left( \bar{r}_{\mu,\guess}^k \!+ \!G_\guess \Delta z^k \right)\!, 
		&\Delta \mu_\inact^k &= W_\inact^{k^{-1}}\! \left( \bar{r}_{\mu,\inact}^k \!+\! G_\inact \Delta z^k \right)\!,
		\label{eq:ls_dual}
		\end{alignat}
\end{subequations}
and the update to the slack variables as $\Delta s^k = - M^{k ^{-1}} (S^k \Delta\mu^k + r_s^k)$.

\subsection{Early Termination based on Duality and Infeasibility Detection}

We do not need to solve a convex QP relaxation in the B\&B method if
\begin{itemize}[noitemsep,topsep=0pt]
	\item the convex QP relaxation is infeasible,
	\item the optimal solution has an objective value that exceeds the current global upper bound.
\end{itemize}
In both cases, the node, and hence the corresponding subtree, can be pruned from the B\&B tree. A considerable computational effort can be avoided if the above scenarios are detected early, i.e., more quickly than the time for solving the convex QP relaxations. Based on our work in~\cite{Liang2021}, we describe a tailored early termination strategy for infeasible primal-dual IPMs to handle both cases and to reduce the computational effort of the B\&B method without affecting the quality of the optimal solution.

Due to duality properties, see, e.g.,~\cite{Boyd2004}, for a dual feasible point $(\mu,\lambda)$ that satisfies~\eqref{eq:dual-eq}-\eqref{ineq:mu} and a primal feasible point $(\v,\y)$ that satisfies~\eqref{eq:QP-G}-\eqref{eq:QP-F}
\begin{equation}
\psi(\mu,\lambda) \le \psi^\star \le \phi^\star \le \phi(\v,\y), \label{eq:duality}
\end{equation}
where $\phi^\star$ and $\psi^\star$ are the primal and dual optima, respectively. Based on~\eqref{eq:duality}, we propose an approach to find a dual feasible point that allows for early termination when $ \psi(\mu,\lambda) > \texttt{UB}$ for the current upper bound~($\texttt{UB}$) to the optimum of the MIQP.




\subsubsection{Projection Strategy for Dual Feasibility}

A dual feasible solution, i.e., $(\mu^k, \lam^k)$ satisfying~\eqref{eq:dual-eq}-\eqref{ineq:mu} is required in order to perform early termination based on the duality result in~\eqref{eq:duality}. Since an infeasible IPM generally does not provide a solution that satisfies the equality constraint in~\eqref{eq:dual-eq} until convergence, we proposed~\cite{Liang2021} a projection step to compute new values $ (\mu^+,\lam^+)=(\mu^k+\Delta \mu, \lam^k +\Delta \lam) $ satisfying~\eqref{eq:dual-eq} and~\eqref{ineq:mu}
by solving the optimization problem
\begin{equation}\label{eq:projection}
\begin{aligned} 
\min_{\Delta \v, \Delta \lam, \Delta \mu} \quad & \frac{1}{2} \|\Delta \v\|_{Q}^2 + \frac{1}{2} \|\Delta \lam\|_{\epsilon_\mathrm{dual}}^2 + \frac{1}{2} \|\Delta \mu\|_{W^k}^2 
\\
\text{s.t.} \quad\;\; & \begin{bmatrix} Q \\0 \end{bmatrix} \Delta \v + F^\top \Delta \lam + G^\top \Delta \mu
= - \begin{bmatrix} 0\\ r_{\y}^k \end{bmatrix},
\end{aligned}
\end{equation}
where $ r_{\y}^k:=F_{\y}^\top \lam^k + G_{\y}^\top \mu^k + h_{\y} $, $W^k = \diag{w^k}$ and $w_i^k = \frac{s_i^k}{\mu_i^k} > 0$. 

There are three advantages for the projection~\eqref{eq:projection} over a standard minimum-norm projection. Neither of these projections directly enforces the positivity constraints $\mu^k + \Delta \mu>0$, which would require solving an inequality constrained QP. A first advantage of~\eqref{eq:projection} is that, since $w_i^k = \frac{s_i^k}{\mu_i^k} > 0$, the term $\|\Delta \mu\|_{W^k}^2 = \sum_i \left( \frac{s_i^k}{\mu_i^k} \Delta \mu_i^2 \right)$ in the objective of~\eqref{eq:projection} penalizes the step $\Delta \mu_i$ to remain small when $\frac{s_i^k}{\mu_i^k}$ is relatively large, i.e., when $\mu_i^k>0$ is close to zero. This makes the step smaller when approaching the positivity constraint, such that it is more likely to satisfy $\mu_i^k + \Delta \mu_i>0$ without making the update step always small. Second, the solution $(\Delta \mu, \Delta \lam)$ to the optimization problem~\eqref{eq:projection} is equivalent~\cite{Liang2021} to solving the symmetric system
\begin{equation}\label{eq:step}
\begin{bmatrix}H & F^\top & G^\top\\F & -\epsilon_{\mathrm{dual}} I & 0 \\G & 0 & -W^k \\\end{bmatrix}
\begin{bmatrix}\Delta \z\\\Delta \lam\\\Delta \mu\end{bmatrix}
=-\begin{bmatrix}\begin{bmatrix}0\\r_{\y}^k\end{bmatrix}\\0\\0\end{bmatrix},
\end{equation}
which is an IPM iteration similar to~\eqref{eq:IPM} with the only differences being the right-hand side and the augmented Lagrangian type regularization $\epsilon_{\mathrm{dual}} > 0$. 
The equivalence means that the same block-tridiagonal matrix factorization for the solution of the reduced linear system~\eqref{eq:ls_primal} in each iteration of \software{ASIPM}~\cite{ASIPM} can be reused to compute the projection step in~\eqref{eq:step}.
Third, the projection~\eqref{eq:step} aims at retaining the IPM progress towards the optimum, since the projection step does not increase the residual value for the remaining optimality conditions in~\eqref{def:rhs}, due to the zero elements in the right-hand side of~\eqref{eq:step}. 

\subsubsection{Early Termination Strategy}

We use the projection step in~\eqref{eq:step} for our early termination strategy, see Algorithm~\ref{alg:earlyTermination}. As discussed later, one dual objective evaluation is computationally cheaper than a projection on a dual feasible point. Therefore, Algorithm~\ref{alg:earlyTermination} performs the projection in~\eqref{eq:step} if and only if the dual objective $\psi(\mu^k,\lam^k)$ is larger than the current UB~(line~5). 
Multiple evaluations of the projection step in~\eqref{eq:step}, reusing the same matrix factorization, may be needed to ensure dual feasibility $\| F_{\y}^\top \lam + G_{\y}^\top \mu + h_{\y} \|< tol$ when using the inexact Newton implementation of~\cite{ASIPM} \change{in the \software{ASIPM} solver}.


	\begin{algorithm}[h]
		\caption{Early termination for IPM in B\&B method.}
		\label{alg:earlyTermination}
		\begin{algorithmic}[1]
			\Require Warm start $\left\{\left(\z^{0}, \mu^{0}, \lambda^{0}, s^{0}\right)\right\}$, $ tol $, and \texttt{UB}.
			\While {$ \max\{\tau^k, \|r^k\|\} > tol $}
			\If {$\psi(\mu^{k},\lam^{k}) > \texttt{UB} \;\;\;\land\;\;$ \texttt{dual\_feasible}}
			\State {\bf break} while loop. \Comment{Early termination}
			\ElsIf {$\psi(\mu^{k},\lam^{k}) > \texttt{UB}$}
			\State Compute projection step $ (\Delta \mu, \Delta \lam) $ in~\eqref{eq:step}.
			\State $\mu \gets \mu^k + \Delta \mu, \, \lam \gets \lam^k + \Delta \lam$, and $r_{\y} \gets F_{\y}^\top \lam + G_{\y}^\top \mu + h_{\y}$.
			\If {$ \mu>0 \;\;\land\;\; \|r_{\y}\|< tol$}
			\State $\mu^k \gets \mu$, $\lam^k \gets \lam$, $r_{\y}^k \gets r_{\y}$, and \texttt{dual\_feasible} $\gets 1$. 
			\IfThen {$\psi(\mu^{k},\lam^{k}) > \texttt{UB}$}{{\bf break} while loop.} \Comment{Early termination}
			\EndIf
			\EndIf
			\State Perform an IPM iteration~\eqref{eq:IPM}, e.g., see~\cite{ASIPM}.
			\EndWhile
		\end{algorithmic}
	\end{algorithm}

The standard iterates of an IPM can also be used to generate certificates of infeasibility~\cite{Todd04detectinginfeasibility}. 
Proposition~\ref{prop:main}, proved in~\cite{Liang2021}, states conditions such that the dual objective $\psi(\mu^k,\lam^k)$ is unbounded in the limit of the primal-dual IPM iterations. Therefore, our proposed early termination strategy is effective for infeasibility detection and, given a tight UB value from the B\&B optimization method, it may lead to termination even before a certificate of infeasibility can be found.
\begin{proposition}\label{prop:main}
	If the sequence of primal-dual iterates $\left\lbrace\left(\z^k, \mu^k, \lambda^k, s^k\right)\right\rbrace$ of the IPM
	satisfy $ \mu^{k^\top} s^k \le {\mu^0}^\top s^0 $ and $ \|\mu^k\|\to \infty $, then the dual objective $\psi(\mu^k,\lam^k) \rightarrow \infty$.
\end{proposition} \vspace{-6mm}

\subsubsection{Computational Complexity}\label{subsec:complexity}

The proposed early termination strategy requires two computational steps, i.e., the projection and the dual objective evaluation, which are typically not needed in a standard IPM. Considering the optimal control structured program~\eqref{eq:OCP-MIQP}, the evaluation of the dual objective value~\eqref{eq:dual} requires $N (n^2+2\,n\,m+2\,n\,p)$
operations to compute $L^{-1} \left( G_{\v}^\top \mu + F_{\v}^\top \lam \right)$ based on a block-diagonal Cholesky factorization $Q = L L^\top$, in which $n$, $m$ and $p$ are the number of variables in~$\v$ (see~\eqref{eq:QP}), the number of equality and inequality constraints, respectively, per control interval, and $N$ is the number of control intervals.

Based on~\eqref{eq:step}, we can perform one projection step at the computational cost of one IPM iteration. 
This allows for reusing the corresponding matrix factorization in the subsequent IPM iteration if the projection is not successful. Based on the particular ASIPM implementation from Section~\ref{subsec:activeSet}, as originally proposed in~\cite{ASIPM}, one iteration requires a block-tridiagonal Cholesky factorization for the matrix in~\eqref{eq:ls_primal}, for which the dominant terms in the computational cost are
\begin{equation}
N\left(\frac{7}{3} \nx^3+4 \nx^2 \nU+2 \nx \nU^2+\frac{1}{3} \nU^3\right), \label{eq:ASIPM_cost}
\end{equation}
where $\nx$ and $\nU$ denote the number of state and control variables per interval in~\eqref{eq:OCP-MIQP}, respectively. 
Then, reusing this matrix factorization, the linear system for the projection step in~\eqref{eq:step} can be solved by
\begin{equation}
\begin{aligned}
N\left(6\,\nx^2 + 8\,\nx \nU + 2\,\nU^2 + 2(\nx+\nU)(m+p)\right), 
\end{aligned}
\label{eq:projection_cost}
\end{equation}
operations for the resulting block-structured forward and backward substitution.
Since $n \le (\nx+\nU)$, and often $n \ll (\nx+\nU)$ due to many auxiliary variables in hybrid systems~\cite{bemporad1999control} for which the Hessian contribution is zero, the cost for a dual objective evaluation is considerably smaller than the projection cost~\eqref{eq:projection_cost}, as anticipated.


\section{Exact Presolve Techniques for Mixed-Integer Optimal Control} \label{sec:PRESOLVE}

A presolve routine is a collection of computationally efficient operations that should be used in each node of the B\&B method (see Figure~\ref{fig:BB_ilu}) before solving the convex relaxation, in order to remove decision variables, and to remove or tighten constraints~\cite{Achterberg2013}.
Presolve techniques are often crucial in strengthening convex relaxations such that typically fewer nodes need to be explored in a B\&B optimization method, sometimes to such an extent that seemingly intractable problems become computationally tractable. We present a collection of tailored variants of presolve techniques with block-sparse structure exploitation for mixed-integer optimal control, based on presolve routines in state-of-the-art solvers for general-purpose MIQPs, e.g., see~\cite{AchterbergBixbyEtAl2019}. 
We explicitly refer to these presolve techniques as \emph{exact} operations, differentiating from the heuristic approach in Section~\ref{sec:HEURISTICS}, and to emphasize that all presolve methods in the present Section preserve feasibility and optimality, i.e., the reduced problem is infeasible or unbounded only if the original problem is infeasible or unbounded, and any feasible or optimal solution of the reduced problem can be mapped to a feasible or optimal solution of the original problem. For example, an exact presolve method will fix a binary variable to $0$ or $1$ only if the method can guarantee that this variable is fixed to that same value in an optimal solution.

\subsection{Block-Structured Domain Propagation}

Several strengthening techniques are implemented as part of ``presolve'' routines in state-of-the-art commercial solvers~\cite{AchterbergBixbyEtAl2019}. One technique that is particularly suitable to mixed-integer optimal control is based on \emph{domain propagation}, in which the goal is to strengthen bound values based on the constraints of the MIQP in~\eqref{eq:OCP-MIQP}. In previous work~\cite{Hespanhol2019}, we suggested to apply domain propagation to the inequality constraints of a condensed MIOCP formulation, i.e., to the smaller but dense problem formulation after numerically eliminating each of the state variables based on the state dynamic constraints in~\eqref{OCP:dyn}. However, based on a simple illustrative example, we show that it can often be advantageous to apply each of the presolve operations, and domain propagation in particular, directly to the block-sparse \change{MIQP} formulation in~\eqref{eq:OCP-MIQP}, including the state variables and state dynamic equality constraints~\eqref{OCP:dyn}. The key insight is that tight lower and upper bounds for state variables can be used to tighten lower and upper bounds on control variables and vice versa.
\begin{Example}
\begin{minipage}{.5\linewidth}
\begin{subequations} \label{eq:OCP-MIQP_ex1}
	\begin{alignat}{5}
	\hspace{-3mm}\underset{X,\,U,\,\delta}{\text{min}} \quad &\sum_{i=0}^{2} \frac{1}{2}\begin{bmatrix}	x_i \\ u_i	\end{bmatrix}^\top \begin{bmatrix} 1 & 0 \\ 0 & 1 \end{bmatrix} \begin{bmatrix}	x_i \\ u_i	\end{bmatrix} && \label{OCP_ex1:obj}\\
	\hspace{-3mm}\text{s.t.} \quad\; 
	& x_0 = 0, 	\label{OCP_ex1:initial}\\
	& x_{i+1} = x_i + u_i, && \hspace{0em} \forall i \in \{0,1\}, \label{OCP_ex1:dyn}\\
	& \begin{bmatrix} -1 \\ -1 \end{bmatrix} \leq \begin{bmatrix} x_i \\ u_i \end{bmatrix} \leq \begin{bmatrix} 1 \\ 1 \end{bmatrix}, \quad &&\hspace{0em} \forall i \in \{0,1,2\}, \label{OCP_ex1:bounds} \\
	& x_2 \ge 2 - 10 \delta, \quad &&\hspace{0em} \delta \in \{0,1\}, \label{OCP_ex1:ineq}
	\end{alignat}
\end{subequations} 
\end{minipage}%
\begin{minipage}{.5\linewidth}
\begin{subequations} \label{eq:OCP-MIQP_ex2}
	\begin{alignat}{5}
	\hspace{-3mm}\underset{U,\,\delta}{\text{min}} \quad &\frac{1}{2} \left(\begin{bmatrix}	u_0 \\ u_1	\end{bmatrix}^\top \begin{bmatrix} 3 & 1 \\ 1 & 2 \end{bmatrix} \begin{bmatrix}	u_0 \\ u_1	\end{bmatrix} + u_2^2 \right) && \label{OCP_ex2:obj}\\
	\hspace{-3mm}\text{s.t.} \quad\; 
	& \begin{bmatrix} -1 \\ -1 \\ -1 \end{bmatrix} \leq \begin{bmatrix} u_0 \\ u_1 \\ u_2 \end{bmatrix} \leq \begin{bmatrix} 1 \\ 1 \\ 1 \end{bmatrix}, \quad &&\hspace{0em}  \label{OCP_ex2:bounds1} \\
	& -1 \leq u_0 + u_1 \leq 1, \quad &&\hspace{0em}  \label{OCP_ex2:bounds2} \\
	& u_0 + u_1 \ge 2 - 10 \delta, \quad &&\hspace{0em} \delta \in \{0,1\}. \label{OCP_ex2:ineq}
	\end{alignat}
\end{subequations} 
\end{minipage}

A standard bound strengthening procedure, as explained in~\cite{AchterbergBixbyEtAl2019,Hespanhol2019}, can be applied to each of the inequality constraints in either the block-sparse \change{MIQP} formulation in~\eqref{eq:OCP-MIQP_ex1} or the equivalent but condensed \change{MIQP} in~\eqref{eq:OCP-MIQP_ex2}. For example, the inequality in~\eqref{OCP_ex1:ineq} can be used to compute a bound on the binary variable $\delta \in \{0,1\}$ as follows:
\begin{equation}
10 \delta \ge 2 - x_2 \quad \underset{x_2 \le 1}{\Longrightarrow} \quad 10 \delta \ge 2 - 1 \quad \Longrightarrow \quad \delta \ge \frac{1}{10} \quad \underset{\delta \in \{0,1\}}{\Longrightarrow} \quad \delta = 1,  \nonumber
\end{equation}
resulting in a fixing of the binary optimization variable $\delta = 1$.
Alternatively, using the condensed inequality in~\eqref{OCP_ex2:ineq}, the same procedure leads to
\begin{equation}
10 \delta \ge 2 - u_0 - u_1 \quad \underset{u_0 \le 1, \; u_1 \le 1}{\Longrightarrow} \quad 10 \delta \ge 2 - 1 - 1 \quad \Longrightarrow \quad \delta \ge 0,  \nonumber
\end{equation}
which means that the binary optimization variable $\delta \in \{0,1\}$ cannot be eliminated in this case, even though the condensed \change{MIQP} in~\eqref{eq:OCP-MIQP_ex2} is equivalent to the block-sparse \change{MIQP} formulation in~\eqref{eq:OCP-MIQP_ex1}.
\qed \label{eq:illust_exam}
\end{Example}

\begin{Remark}
In Example~\ref{eq:illust_exam}, for the condensed inequality in~\eqref{OCP_ex2:ineq}, fixing of the binary variable $\delta = 1$ could be detected by exploiting a bound for the expression $u_0 + u_1 \le 1$ from~\eqref{OCP_ex2:bounds2}. However, as discussed in~\cite[Section~5.4]{AchterbergBixbyEtAl2019}, a bound strengthening procedure considering multiple constraints at once is typically too expensive for MIQP solvers, but more general optimization-based bound tightening~(OBBT) techniques are common for mixed-integer nonlinear programming~(MINLP).
\end{Remark}

Motivated by the above illustrative example, we propose a novel block-sparse variant of domain propagation as described in Algorithm~\ref{alg:Sparse_DP_sweep} that is tailored to the \change{MIQP} formulation in~\eqref{eq:OCP-MIQP}. The method consists of a forward iterative procedure for $i=0,1,\ldots,N$~(see Line~1-14), followed by a backward iterative procedure for $i=N-1,\ldots,0$~(see Line~15-21). The general intuition behind the forward-backward propagation is to have a quick propagation of variable bounds from near the beginning of the prediction horizon towards the end of the prediction horizon, as well as a quick propagation of variable bounds from near the end of the prediction horizon towards the beginning. The proposed forward-backward implementation can reduce the amount of times that the domain propagation in Alg.~\ref{alg:Sparse_DP_sweep} needs to be called in order to achieve a particular amount of bound strengthening in practice, even though such a performance improvement cannot be guaranteed in general. However, by design, it can be guaranteed that the updated bound values $[\ubar{z}_i^+, \bar{z}_i^+]$ that are computed by Alg.~\ref{alg:Sparse_DP_sweep} are tighter than the original bound values $[\ubar{z}_i, \bar{z}_i]$, i.e., $\ubar{z}_i \le \ubar{z}_i^+ \le z_i \le \bar{z}_i^+ \le \bar{z}_i$ for $i \in \{0,1,\ldots,N\}$, without eliminating any feasible solution of the \change{MIQP} in~\eqref{eq:OCP-MIQP} when replacing the bound values by the updated values $[\ubar{z}_i^+, \bar{z}_i^+]$ in~\eqref{OCP:bounds}. Algorithm~\ref{alg:Sparse_DP_sweep} can result in strengthening of bound values for both continuous and integer/binary optimization variables. In addition, the tailored block-sparse \change{MIQP} structure exploitation in Alg.~\ref{alg:Sparse_DP_sweep} considerably reduces the computational cost for each operation in the domain propagation.

\begin{algorithm}[h]
	\caption{Block-sparse forward-backward operations for domain propagation and bound strengthening}
	\label{alg:Sparse_DP_sweep}
		\begin{algorithmic}[1]
			\Require Bound values $[\ubar{z}_{i},\bar{z}_{i}]$, $E_i = \left[ C_i \; D_i \right]$, $F_i = \left[ A_i \; B_i \right]$, $i \in \{0,\ldots,N\}$ and \change{MIQP} of the form~\eqref{eq:OCP-MIQP}.
			\For{$i = 0, \ldots, N$} \Comment{Forward domain propagation}
			\For{$j = 1, \ldots, (\nx+\nU)$} \Comment{Affine inequality constraints: $\blo_{i} \leq E_{i}z_{i} \leq \bup_{i}$}
			\State $\ubar{z}_{i,j} \gets \min\{ z_{i,j}: \blo_{i} \leq E_{i}z_{i} \leq \bup_{i},  \;\; \ubar{z}_{i} \leq z_{i} \leq \bar{z}_{i} \}$, see Alg.~\ref{alg:min_max}
			\State $\bar{z}_{i,j} \gets \max\{ z_{i,j}: \blo_{i} \leq E_{i}z_{i} \leq \bup_{i},  \;\; \ubar{z}_{i} \leq z_{i} \leq \bar{z}_{i} \}$, see Alg.~\ref{alg:min_max}
			\IfThen{$\bar{z}_{i,j} - \ubar{z}_{i,j} < -\epsilon$\;}{\Return \texttt{infeasible\_flag}} \Comment{\change{MIQP} subproblem infeasible}
			\EndFor
			\If{$i < N$} 
			\State $\ubar{a}_i \gets a_{i} + \left(F_i^+ \ubar{z}_{i} + F_i^- \bar{z}_{i}\right)$ and $\bar{a}_i \gets a_{i} + \left(F_i^+ \bar{z}_{i} + F_i^- \ubar{z}_{i}\right)$.
			\For{$j = 1, \ldots, \nx$} \Comment{State dynamic constraints: $x_{i+1} = a_{i} + F_{i}z_{i}$}
			\State $\ubar{x}_{i+1,j} \gets \max(\ubar{x}_{i+1,j},\ubar{a}_{i,j})$ and $\bar{x}_{i+1,j} \gets \min(\bar{x}_{i+1,j},\bar{a}_{i,j})$.
			\IfThen{$\bar{x}_{i+1,j} - \ubar{x}_{i+1,j} < -\epsilon$\;}{\Return \texttt{infeasible\_flag}} \Comment{\change{MIQP} subproblem infeasible}
			\EndFor
			\EndIf
			\EndFor
			\For{$i = N-1, \ldots, 0$} \Comment{Backward domain propagation}
			\For{$j = 1, \ldots, (\nx+\nU)$} \Comment{State dynamic constraints: $x_{i+1} = a_{i} + F_{i}z_{i}$}
			\State $\ubar{z}_{i,j} \gets \min\{ z_{i,j}: x_{i+1} - a_{i} - F_{i}z_{i} = 0,  \;\; \ubar{z}_{i} \leq z_{i} \leq \bar{z}_{i} \}$
			\State $\bar{z}_{i,j} \gets \max\{ z_{i,j}: x_{i+1} - a_{i} - F_{i}z_{i} = 0,  \;\; \ubar{z}_{i} \leq z_{i} \leq \bar{z}_{i} \}$
			\IfThen{$\bar{z}_{i,j} - \ubar{z}_{i,j} < -\epsilon$\;}{\Return \texttt{infeasible\_flag}} \Comment{\change{MIQP} subproblem infeasible}
			\EndFor
			\EndFor
			\Ensure Updated bound values $[\ubar{z}_{i}^+,\bar{z}_{i}^+], \; i \in \{0,\ldots,N\}$.
		\end{algorithmic}
\end{algorithm}

Each iteration of the forward procedure for $i=0,1,\ldots,N$~(Line~1-14) performs domain propagation for each variable $z_{i,j}$ for $j=1,\ldots,\nx+\nU$ based on the affine inequality constraints $\blo_{i} \leq E_{i}z_{i} \leq \bup_{i}$~(see Line~2-6), followed by domain propagation for each state variable $x_{i+1,j}$ for $j=1,\ldots,\nx$ based on the state dynamic equality constraints $x_{i+1} = a_{i} + F_{i}z_{i}$~(see Line~7-13). The computation of $\ubar{z}_{i,j}$ and $\bar{z}_{i,j}$ on Line~3 and~4, respectively, is defined using a compact OBBT notation. However, as described in the next section, we instead perform a single-row approximation for each inequality constraint individually, in order to reduce the computational cost for each iteration of the domain propagation. An \change{MIQP} subproblem is detected to be infeasible whenever the gap between a lower and upper bound value is below a particular threshold value $-\epsilon$, where $\epsilon$ is a small positive value, for example, see Lines~5, 11 and 19. The computation of $\ubar{a}_i = \min\{ a_{i} + F_{i}z_{i} \} = a_{i} + \left(F_i^+ \ubar{z}_{i} + F_i^- \bar{z}_{i}\right)$ and $\bar{a}_i = \max\{ a_{i} + F_{i}z_{i} \} = a_{i} + \left(F_i^+ \bar{z}_{i} + F_i^- \ubar{z}_{i}\right)$ on Line~8 is used for domain propagation based on the state dynamics, where $F_i^+, F_i^-$ contain all positive and negative elements of the matrix $F_i = F_i^+ + F_i^-$, respectively. Each iteration of the backward procedure $i=N-1,\ldots,0$~(see Line~15-21) performs domain propagation for each variable $z_{i,j}$ for $j=1,\ldots,\nx+\nU$ based on the state dynamic equality constraints $x_{i+1} = a_{i} + F_{i}z_{i}$. Similarly, OBBT is used to update the bound values $\ubar{z}_{i,j}$ and $\bar{z}_{i,j}$ on Line~17 and~18, respectively, but a computationally efficient single-row approximation can be used.

\subsection{Approximation of Optimization-based Bound Tightening}

The computation of bound values $\ubar{z}_{i,j}$, $\bar{z}_{i,j}$ on Lines 3-4 and Lines 17-18 of Alg.~\ref{alg:Sparse_DP_sweep} require the solution of a linear programming~(LP) problem. Since this operation needs to be performed for each variable on each time step in the horizon and for each domain propagation call in each iteration of the presolve routine, it is necessary to perform a computationally cheap single-row approximation instead as implemented in Alg.~\ref{alg:min_max} based on the following lemma.

\begin{Lemma}
The updated bound values $\ubar{z}_{i,j}^+$, $\bar{z}_{i,j}^+$ that are computed by Alg.~\ref{alg:min_max} are a conservative approximation of optimization-based bound tightening, i.e., the following relationships hold
\begin{subequations} \label{eq:OBBT1}
	\begin{alignat}{5}
\ubar{z}_{i,j}^+ &= \underset{k \in \{1,\ldots,\nc\}}{\max}\left( \min\{ z_{i,j}: \blo_{i,k} \leq E_{i,k,:}z_{i} \leq \bup_{i,k},  \;\; \ubar{z}_{i} \leq z_{i} \leq \bar{z}_{i} \} \right) \label{eq:OBBT1-1} \\
&\le \; \min\{ z_{i,j}: \blo_{i} \leq E_{i}z_{i} \leq \bup_{i},  \;\; \ubar{z}_{i} \leq z_{i} \leq \bar{z}_{i} \} \quad\le\; \min\{ z_{i,j}: \eqref{OCP:dyn}-\eqref{OCP:int} \}, \label{eq:OBBT1-2}
\end{alignat}
\end{subequations}
where $E_{i,k,:}$ denotes the $k^{\text{th}}$ row of the matrix $E_i$, and similarly for the updated upper bound values
\begin{subequations} \label{eq:OBBT2}
	\begin{alignat}{5}
	\bar{z}_{i,j}^+ &= \underset{k \in \{1,\ldots,\nc\}}{\min}\left( \max\{ z_{i,j}: \blo_{i,k} \leq E_{i,k,:}z_{i} \leq \bup_{i,k},  \;\; \ubar{z}_{i} \leq z_{i} \leq \bar{z}_{i} \} \right) \label{eq:OBBT2-1} \\
	&\ge \; \max\{ z_{i,j}: \blo_{i} \leq E_{i}z_{i} \leq \bup_{i},  \;\; \ubar{z}_{i} \leq z_{i} \leq \bar{z}_{i} \} \quad\ge\; \max\{ z_{i,j}: \eqref{OCP:dyn}-\eqref{OCP:int} \}. \label{eq:OBBT2-2}
	\end{alignat}
\end{subequations}
\end{Lemma}
\begin{proof}
The relationships in~\eqref{eq:OBBT1-2} follow directly from the following inequalities
\begin{equation*}
\min\{ z_{i,j}: \blo_{i,k} \leq E_{i,k,:}z_{i} \leq \bup_{i,k},  \;\; \ubar{z}_{i} \leq z_{i} \leq \bar{z}_{i} \} \;\; \le \;\; \min\{ z_{i,j}: \blo_{i} \leq E_{i}z_{i} \leq \bup_{i},  \;\; \ubar{z}_{i} \leq z_{i} \leq \bar{z}_{i} \} \;\;\le\;\; \min\{ z_{i,j}: \eqref{OCP:dyn}-\eqref{OCP:int} \},
\end{equation*}
which holds for each $k \in \{1,\ldots,\nc\}$, due to the fact that an optimal objective value of a minimization problem only decreases after removing one or multiple constraints. The relationships in~\eqref{eq:OBBT2-2} can be proved by a similar argument. 
\end{proof}

A compact notation is used in Alg.~\ref{alg:min_max} for matrices $E_i^+, E_i^-$ which contain all positive and negative elements of the matrix $E_i = E_i^+ + E_i^-$, respectively. Line~10 of Alg.~\ref{alg:min_max} includes rounding up $\lceil \ubar{d} \rceil$ and rounding down $\lfloor \bar{d} \rfloor$ of the lower and upper bound values, respectively, if $z_{i,j}$ is an integer or binary optimization variable.

\begin{algorithm}[h]
	\caption{Single-row approximation of optimization-based bound tightening for variable $z_{i,j}$}
	\label{alg:min_max}
		\begin{algorithmic}[1]
			\Require Bound values $[\ubar{z}_{i},\bar{z}_{i}]$, $E_i = \left[ C_i \; D_i \right]$, $i \in \{0,\ldots,N\}$, $j \in \{1, \ldots, n_{\mathrm{z}}\}$ and \change{MIQP} of the form~\eqref{eq:OCP-MIQP}.
			\State $\tilde{\bup} \gets \bup_{i} - \left(E_i^+ \ubar{z}_{i} + E_i^- \bar{z}_{i}\right)$ and $\tilde{\blo} \gets \blo_{i} - \left(E_i^+ \bar{z}_{i} + E_i^- \ubar{z}_{i}\right)$.
			\State $\ubar{d} \gets -\infty$ and $\bar{d} \gets \infty$.
			\For{$k = 1, \ldots, \nc$} \Comment{Affine inequality constraints: $\blo_{i} \leq E_{i}z_{i} \leq \bup_{i}$}
			\If{$E_{i,k,j} > \epsilon$}
			\State $\ubar{d}_k \gets \frac{\tilde{\blo}_k + E_{i,k,j} \bar{z}_{i,j}}{E_{i,k,j}}$ and $\bar{d}_k \gets \frac{\tilde{\bup}_k + E_{i,k,j} \ubar{z}_{i,j}}{E_{i,k,j}}$.
			\ElsIf{$E_{i,k,j} < -\epsilon$}
			\State $\ubar{d}_k \gets \frac{\tilde{\bup}_k + E_{i,k,j} \bar{z}_{i,j}}{E_{i,k,j}}$ and $\bar{d}_k \gets \frac{\tilde{\blo}_k + E_{i,k,j} \ubar{z}_{i,j}}{E_{i,k,j}}$.
			\EndIf
			\EndFor
			\IfThen{$z_{i,j}$ is integer}{$\ubar{d} \gets \lceil \ubar{d} \rceil$ and $\bar{d} \gets \lfloor \bar{d} \rfloor$.} \Comment{Rounding integer variable bounds}
			\State $\ubar{z}_{i,j}^+ \gets \max(\ubar{z}_{i,j},\underset{k}{\max}(\ubar{d}_k))$ and $\bar{z}_{i,j}^+ \gets \min(\bar{z}_{i,j},\underset{k}{\min}(\bar{d}_k))$.
			\Ensure $\ubar{z}_{i,j}^+ = \underset{k \in \{1,\ldots,\nc\}}{\max}\left( \min\{ z_{i,j}: \blo_{i,k} \leq E_{i,k,:}z_{i} \leq \bup_{i,k},  \;\; \ubar{z}_{i} \leq z_{i} \leq \bar{z}_{i} \} \right)$,
			\Statex \hspace{7mm} $\bar{z}_{i,j}^+ = \underset{k \in \{1,\ldots,\nc\}}{\min}\left( \max\{ z_{i,j}: \blo_{i,k} \leq E_{i,k,:}z_{i} \leq \bup_{i,k},  \;\; \ubar{z}_{i} \leq z_{i} \leq \bar{z}_{i} \} \right)$.
		\end{algorithmic}
\end{algorithm}



\subsection{Trivial Constraints and Dual Fixings}

Algorithm~\ref{alg:Process_Inequalities} illustrates how single-row bounding for each of the affine inequality constraints can be used to detect an \change{MIQP} subproblem to be infeasible~(see Line~3), for example, if
\begin{equation}
\min\{ E_{i,k,:}z_{i}:  \;\; \ubar{z}_{i} \leq z_{i} \leq \bar{z}_{i} \} = E_{i,k,:}^+ \ubar{z}_{i} + E_{i,k,:}^- \bar{z}_{i} > \bup_{i,k},
\end{equation}
where $E_{i,k,:}$ denotes the $k^{\text{th}}$ row of the matrix $E_i$, and similarly for the lower bound of $\blo_{i,k} \leq E_{i,k,:}z_{i} \leq \bup_{i,k}$, $k \in \{1,\ldots,\nc\}$. In addition, single-row bounding can be used to detect lower and/or upper bounds of an affine inequality constraint to be redundant~(see Line~5-6), for example,
\begin{equation}
\bup_{i,k} \ge \max\{ E_{i,k,:}z_{i}:  \;\; \ubar{z}_{i} \leq z_{i} \leq \bar{z}_{i} \} = E_{i,k,:}^+ \bar{z}_{i} + E_{i,k,:}^- \ubar{z}_{i} \qquad \Longrightarrow \quad \bup_{i,k} \gets \infty \quad \text{(upper bound redundant)},
\end{equation}
and similarly for the lower bound of each affine inequality constraint $\blo_{i,k} \leq E_{i,k,:}z_{i} \leq \bup_{i,k}$, $k \in \{1,\ldots,\nc\}$. Finally, if an optimization variable does not enter the state dynamic equality constraints~(see Line~9 of Alg.~\ref{alg:Process_Inequalities}) and if the variable enters each of the non-redundant inequality constraints with a coefficient of the same sign, then the variable can be fixed to the lower~(Line~10) or the upper bound~(Line~11) value, depending on the sign of the constraint coefficients and the sign in the cost function. Lines~7-13 of Alg.~\ref{alg:Process_Inequalities} form a tailored block-sparse variant of \emph{dual fixing}, which is described more generally in~\cite{AchterbergBixbyEtAl2019}.

	\begin{algorithm}[h]
		\caption{Block-sparse redundant inequality constraint detection and dual fixings}
		\label{alg:Process_Inequalities}
		\begin{algorithmic}[1]
			\Require Bound values $[\ubar{z}_{i},\bar{z}_{i}]$, $E_i = \left[ C_i \; D_i \right]$, $[\blo_{i},\bup_{i}], \; i \in \{0,\ldots,N\}$ and \change{MIQP} of the form~\eqref{eq:OCP-MIQP}.
			\For{$i = 0, \ldots, N$} 
			\State $\tilde{\bup} \gets \bup_{i} - \left(E_i^+ \ubar{z}_{i} + E_i^- \bar{z}_{i}\right)$ and $\tilde{\blo} \gets \blo_{i} - \left(E_i^+ \bar{z}_{i} + E_i^- \ubar{z}_{i}\right)$. \Comment{Inequality constraints: $\blo_{i} \leq E_{i}z_{i} \leq \bup_{i}$}
			\IfThen{$\underset{k}{\min}(\tilde{\bup}_k) < -\epsilon \; \lor \; \underset{k}{\max}(\tilde{\blo}_k) > \epsilon$\;}{\Return \texttt{infeasible\_flag}} \Comment{\change{MIQP} subproblem infeasible}
			\State $\tilde{\bup} \gets \bup_{i} - \left(E_i^+ \bar{z}_{i} + E_i^- \ubar{z}_{i}\right)$ and $\tilde{\blo} \gets \blo_{i} - \left(E_i^+ \ubar{z}_{i} + E_i^- \bar{z}_{i}\right)$.
			\State $\bup_{i,k} \gets \infty$ \textbf{if} $\tilde{\bup}_k > -\epsilon, \quad \forall k = 1, \ldots, \nc$ \Comment{Upper bound constraint redundant}
			\State $\blo_{i,k} \gets -\infty$ \textbf{if} $\tilde{\blo}_k < \epsilon, \quad \forall k = 1, \ldots, \nc$ \Comment{Lower bound constraint redundant}
			\For{$j = 1, \ldots, \nU$}
			\State $\tilde{D} \gets \left[ \{ D_{i,k,j} \}_{\forall k: \bup_{i,k} \,<\, \epsilon_{\mathrm{max}}} \quad \{-D_{i,k,j} \}_{\forall k: \blo_{i,k} \,>\, -\epsilon_{\mathrm{max}}} \right]$. \Comment{Non-redundant inequality constraints}
			\If{$\Vert B_{i, :,j} \Vert < \epsilon \; \land \; \Vert H_{i, :,\nx+j} \Vert < \epsilon \; \land \; \Vert \bar{u}_{i,j} - \ubar{u}_{i,j} \Vert > \epsilon$}
			\IfThen{$\underset{k}{\min}(\tilde{D}_k) > -\epsilon \; \land \; r_{i,j} > -\epsilon$}{$\bar{u}_{i,j} \gets \ubar{u}_{i,j}$} \Comment{Fixing auxiliary variable to lower bound}
			\IfThen{$\underset{k}{\max}(\tilde{D}_k) < \epsilon \; \land \; r_{i,j} < \epsilon$}{$\ubar{u}_{i,j} \gets \bar{u}_{i,j}$} \Comment{Fixing auxiliary variable to upper bound}
			\EndIf
			\EndFor
			\EndFor
			\Ensure Updated bound values $[\ubar{z}_{i},\bar{z}_{i}]$ and $[\blo_{i},\bup_{i}], \; i \in \{0,\ldots,N\}$.
		\end{algorithmic}
	\end{algorithm}


\subsection{Constraint Coefficient Strengthening}

In \emph{coefficient strengthening}, we aim to modify coefficients of the affine inequality constraints~\eqref{OCP:ineq} to tighten the convex QP relaxation without affecting the integer-feasible \change{MIQP} solutions. We present a block-sparse variant of coefficient strengthening in Alg.~\ref{alg:Coefficient_Strengthening}, as described more generally in~\cite{Savelsbergh1994}. Let us define constraint domination as follows:
\begin{Definition}[see~\cite{AchterbergBixbyEtAl2019,Savelsbergh1994}]
Given two inequality constraints $e_1^\top z \le f_1$ and $e_2^\top z \le f_2$, where $e_1^\top$ and $e_2^\top$ denote row vectors, $z \in \R^{n}$ for which $z_{j} \in \Z, \; \forall j \in \I$ and $\ubar{z} \le z \le \bar{z}$, then $e_1^\top z \le f_1$ \emph{dominates} $e_2^\top z \le f_2$ if
\begin{equation}
\{ z \in \R^{n} \; | \; \ubar{z} \le z \le \bar{z}, \; e_1^\top z \le f_1 \} \quad \subset \quad \{ z \in \R^{n} \; | \; \ubar{z} \le z \le \bar{z}, \; e_2^\top z \le f_2 \},
\end{equation}
i.e., the feasible region of the convex QP relaxation for $e_1^\top z \le f_1$ is strictly contained in the feasible region for $e_2^\top z \le f_2$, given the variable bound values $\ubar{z} \le z \le \bar{z}$.
\end{Definition}	
Using the latter definition, the overall aim of coefficient strengthening is to reformulate one or multiple affine inequality constraints~\eqref{OCP:ineq} as constraints that \emph{dominate} the original inequality constraints, without removing any feasible \change{MIQP} solution.
We demonstrate the idea of coefficient strengthening based on following simple illustrative example.
\begin{Example}
	\label{ex:coeffStrength}
Let us consider an MIQP with a continuous variable $x \in \R$, which is bounded $1 \le x \le 3$, and a binary variable $\delta \in \{0,1\}$. Coefficient strengthening for an inequality constraint $2 \le x + 100\, \delta$ then results in a dominating constraint $2 \le x + \delta$, i.e., 
\begin{equation}
\{ x \in \R, \,\delta \in \R \; | \; 1 \le x \le 3, \; 0 \le \delta \le 1, \; 2 \le x + \delta \} \quad \subset \quad \{ x \in \R, \,\delta \in \R \; | \; 1 \le x \le 3, \; 0 \le \delta \le 1, \; 2 \le x + 100\, \delta \}. \label{ex:bigM}
\end{equation}
\end{Example}
Mixed-integer inequality constraints of the form in~\eqref{ex:bigM} are common, e.g., when using a ``big-M'' formulation~\cite{Trespalacios2015} in MIOCPs, such that coefficient strengthening can be used to automatically reduce the large coefficient value $M>0$ and tighten the convex QP relaxations.
Algorithm~\ref{alg:Coefficient_Strengthening} describes a systematic approach to strengthen the coefficients of each affine inequality constraint $\blo_{i,k} \leq E_{i,k,:}z_{i} \leq \bup_{i,k}$, $k \in \{1,\ldots,\nc\}$, with respect to each integer or binary optimization variable $u_{i,j}, \, \forall j \in \I_i$ in~\eqref{OCP:int}.

\begin{algorithm}[h]
	\caption{Block-sparse inequality constraint coefficient strengthening}
	\label{alg:Coefficient_Strengthening}
		\begin{algorithmic}[1]
			\Require Bound values $[\ubar{z}_{i},\bar{z}_{i}]$, $E_i = \left[ C_i \; D_i \right]$, $[\blo_{i},\bup_{i}], \; i \in \{0,\ldots,N\}$ and \change{MIQP} of the form~\eqref{eq:OCP-MIQP}.
			\For{$i = 0, \ldots, N$} 
			\State $\tilde{\bup} \gets \bup_{i} - \left(E_i^+ \bar{z}_{i} + E_i^- \ubar{z}_{i}\right)$ and $\tilde{\blo} \gets \left(E_i^+ \ubar{z}_{i} + E_i^- \bar{z}_{i}\right) - \blo_{i}$. \Comment{Inequality constraints: $\blo_{i} \leq E_{i}z_{i} \leq \bup_{i}$}
			\For{$j = 1, \ldots, \nU$ \textbf{and if} $u_{i,j}$ is integer $\land \; \Vert \bar{u}_{i,j} - \ubar{u}_{i,j} \Vert > \epsilon$ \textbf{then}}
			\For{$k = 1, \ldots, \nc$}
			\If{$\bup_{i,k} \,<\, \epsilon_{\mathrm{max}} \; \land \; \blo_{i,k} \,<\, -\epsilon_{\mathrm{max}}$}
			\State $\tilde{d} \gets \tilde{\bup}_k + | D_{i,k,j} | \left( \bar{u}_{i,j} - \ubar{u}_{i,j} \right)$.
			\If{$\tilde{d} > \epsilon \; \land \; D_{i,k,j} > \epsilon$}
			\State $D_{i,k,j} \gets D_{i,k,j} - \tilde{d}$ and $\bup_{i,k} \gets \bup_{i,k} - \tilde{d} \, \bar{u}_{i,j}$. \Comment{Strengthening of coefficient $D_{i,k,j}$}
			\ElsIf{$\tilde{d} > \epsilon \; \land \; D_{i,k,j} < -\epsilon$}
			\State $D_{i,k,j} \gets D_{i,k,j} + \tilde{d}$ and $\bup_{i,k} \gets \bup_{i,k} + \tilde{d} \, \ubar{u}_{i,j}$. \Comment{Strengthening of coefficient $D_{i,k,j}$}
			\EndIf
			\ElsIf{$\blo_{i,k} \,>\, -\epsilon_{\mathrm{max}} \; \land \; \bup_{i,k} \,>\, \epsilon_{\mathrm{max}}$}
			\State $\tilde{d} \gets \tilde{\blo}_k + | D_{i,k,j} | \left( \bar{u}_{i,j} - \ubar{u}_{i,j} \right)$.
			\If{$\tilde{d} > \epsilon \; \land \; D_{i,k,j} > \epsilon$}
			\State $D_{i,k,j} \gets D_{i,k,j} - \tilde{d}$ and $\blo_{i,k} \gets \blo_{i,k} - \tilde{d} \, \ubar{u}_{i,j}$. \Comment{Strengthening of coefficient $D_{i,k,j}$}
			\ElsIf{$\tilde{d} > \epsilon \; \land \; D_{i,k,j} < -\epsilon$}
			\State $D_{i,k,j} \gets D_{i,k,j} + \tilde{d}$ and $\blo_{i,k} \gets \blo_{i,k} + \tilde{d} \, \bar{u}_{i,j}$. \Comment{Strengthening of coefficient $D_{i,k,j}$}
			\EndIf
			\EndIf
			\EndFor
			\EndFor
			\EndFor
			\Ensure Updated constraint bound values $[\blo_{i},\bup_{i}]$ and matrices $D_i, \; i \in \{0,\ldots,N\}$.
		\end{algorithmic}
\end{algorithm}

\subsection{Binary Variable Probing}

The general idea of probing is to select a binary variable, which is set tentatively to zero or one in order to derive further variable fixings and/or tightened inequality constraints, see~\cite{AchterbergBixbyEtAl2019,Savelsbergh1994}. For example, let $u_{i,j} \in \{0,1\}$ be a binary variable, and let us define the lower and upper bounds $\ubar{z}_{i}^0 \leq z_{i} \leq \bar{z}_{i}^0$ for $i=0,\ldots,N$ that have been deduced from setting $u_{i,j}=0$ using one or multiple iterations of presolve operations, e.g., using the forward-backward domain propagation in Alg.~\ref{alg:Sparse_DP_sweep}. Similarly, the lower and upper bounds $\ubar{z}_{i}^1 \leq z_{i} \leq \bar{z}_{i}^1$ for $i=0,\ldots,N$ have been deduced from setting $u_{i,j}=1$. Our tailored implementation of binary variable probing in Alg.~\ref{alg:Probing} is based on following observations:
	\begin{itemize}
		\item If both $u_{i,j}=0$ and $u_{i,j}=1$ leads to an infeasible problem $\quad\Longrightarrow\quad$ problem is infeasible
		\item If $u_{i,j}=0$ leads to infeasible problem $\quad\Longrightarrow\quad$ $u_{i,j}=1$ and $\ubar{z}_{i,k} = \ubar{z}_{i,k}^1$, $\bar{z}_{i,k} = \bar{z}_{i,k}^1$, $i=0,\ldots,N, \; k = 1, \ldots, (\nx+\nU)$
		\item If $u_{i,j}=1$ leads to infeasible problem $\quad\Longrightarrow\quad$ $u_{i,j}=0$ and $\ubar{z}_{i,k} = \ubar{z}_{i,k}^0$, $\bar{z}_{i,k} = \bar{z}_{i,k}^0$, $i=0,\ldots,N, \; k = 1, \ldots, (\nx+\nU)$
		\item Bound values updated for each variable $\ubar{z}_{i,k} = \min\{\ubar{z}_{i,k}^0,\ubar{z}_{i,k}^1\}$ and $\bar{z}_{i,k} = \max\{\bar{z}_{i,k}^0,\bar{z}_{i,k}^1\}$, $i=0,\ldots,N, \; k = 1, \ldots, (\nx+\nU)$
		\item If $\exists k: \ubar{z}_{i,k}^0 = \bar{z}_{i,k}^0$ and $\ubar{z}_{i,k}^1 = \bar{z}_{i,k}^1$ $\quad\Longrightarrow\quad$ $z_{i,k}$ can be substituted as $z_{i,k} = \bar{z}_{i,k}^0+(\bar{z}_{i,k}^1-\bar{z}_{i,k}^0) \,u_{i,j}$
	\end{itemize}
For simplicity, the last observation is not included in Alg.~\ref{alg:Probing}.
Binary variable probing is a relatively simple strategy that can be very effective at reducing the B\&B search tree, but it can become computationally expensive. Therefore, a limit on the number of probing iterations, $n < n_{\mathrm{probing}}$~(see Line~5 and Line~11), and/or a timeout is needed to ensure computational efficiency.

\begin{algorithm}[h]
	\caption{Block-sparse binary variable probing procedure}
	\label{alg:Probing}
		\begin{algorithmic}[1]
			\Require Bound values $[\ubar{z}_{i},\bar{z}_{i}], \; i \in \{0,\ldots,N\}$ and \change{MIQP} of the form~\eqref{eq:OCP-MIQP}.
			\For{$i = 0, \ldots, N$} 
			\For{$j = 1, \ldots, \nU$ \textbf{and if} $u_{i,j}$ is binary $\land \; \Vert \bar{z}_{i,\nx+j} - \ubar{z}_{i,\nx+j} \Vert > \epsilon$ \textbf{then}}
			\State $[\ubar{z}^0,\bar{z}^0] \gets [\ubar{z},\bar{z}]$ and $\bar{z}^0_{i,\nx+j} \gets \ubar{z}^0_{i,\nx+j}$. \Comment{Probing variable $u_{i,j}$ to lower bound}
			\State \texttt{infeasible}\_0 $\gets$ \texttt{false} and $n \gets 0$.
			\While{\texttt{sufficient\_progress} $\; \land \;$ !\texttt{infeasible}\_0 $\; \land \;$ $n < n_{\mathrm{probing}}$ }
			\State Update bound values $[\ubar{z}^0_{i},\bar{z}^0_{i}], \; i \in \{0,\ldots,N\}$ using forward-backward process in Alg.~\ref{alg:Sparse_DP_sweep}.
			\State \texttt{infeasible}\_0 $\gets$ \texttt{infeasible\_flag} for $[\ubar{z}^0,\bar{z}^0]$ and $n \gets n+1$.
			\EndWhile
			\State $[\ubar{z}^1,\bar{z}^1] \gets [\ubar{z},\bar{z}]$ and $\ubar{z}^1_{i,\nx+j} \gets \bar{z}^1_{i,\nx+j}$. \Comment{Probing variable $u_{i,j}$ to upper bound}
			\State \texttt{infeasible}\_1 $\gets$ \texttt{false} and $n \gets 0$.
			\While{\texttt{sufficient\_progress} $\; \land \;$ !\texttt{infeasible}\_1 $\; \land \;$ $n < n_{\mathrm{probing}}$ }
			\State Update bound values $[\ubar{z}^1_{i},\bar{z}^1_{i}], \; i \in \{0,\ldots,N\}$ using forward-backward process in Alg.~\ref{alg:Sparse_DP_sweep}.
			\State \texttt{infeasible}\_1 $\gets$ \texttt{infeasible\_flag} for $[\ubar{z}^1,\bar{z}^1]$ and $n \gets n+1$.
			\EndWhile
			\If{\texttt{infeasible}\_0 $\quad \land \quad$ \texttt{infeasible}\_1}
			\State \Return \texttt{infeasible\_flag} \Comment{QP subproblem infeasible}
			\ElsIf{\texttt{infeasible}\_0}
			\State $[\ubar{z},\bar{z}] \gets [\ubar{z}^1,\bar{z}^1]$. \Comment{Fix variable $u_{i,j}$ to upper bound and update all values}
			\ElsIf{\texttt{infeasible}\_1}
			\State $[\ubar{z},\bar{z}] \gets [\ubar{z}^0,\bar{z}^0]$. \Comment{Fix variable $u_{i,j}$ to lower bound and update all values}
			\Else
			\State $[\ubar{z},\bar{z}] \gets [\min\{\ubar{z}^0,\ubar{z}^1\},\max\{\bar{z}^0,\bar{z}^1\}]$. \Comment{Update all bound values}
			\EndIf
			\EndFor
			\EndFor
			\Ensure Updated bound values $[\ubar{z}_{i},\bar{z}_{i}], \; i \in \{0,\ldots,N\}$.
		\end{algorithmic}
\end{algorithm}

\subsection{Exact Block-sparse Presolve Procedure}

Finally, we summarize our tailored block-sparse presolve procedure for optimal control structured MIQPs in Alg.~\ref{alg:Sparse_presolve}, which should be called in each node of the B\&B method (see Figure~\ref{fig:BB_ilu}) before solving the convex relaxation. Each iteration of Alg.~\ref{alg:Sparse_presolve} includes the following block-sparse presolve operations:
	\begin{enumerate}
		\item Update variable bound values $[\ubar{z}_{i},\bar{z}_{i}]$, using forward-backward domain propagation in Alg.~\ref{alg:Sparse_DP_sweep}~(Line~3).
		\item Update bound values $[\ubar{z}_{i},\bar{z}_{i}]$ and $[\blo_{i},\bup_{i}]$, using redundant inequality constraint detection in Alg.~\ref{alg:Process_Inequalities}~(Line~4).
		\item Update bound values $[\blo_{i},\bup_{i}]$ and matrices $D_i$, using constraint coefficient strengthening in Alg.~\ref{alg:Coefficient_Strengthening}~(Line~5).
		\item If probing is enabled and allowed~(see Line~6 of Alg.~\ref{alg:Sparse_presolve}), use binary variable probing in Alg.~\ref{alg:Probing}~(Line~7).
	\end{enumerate}
The block-sparse presolve procedure in Alg.~\ref{alg:Sparse_presolve} is an iterative procedure that typically requires multiple iterations, because each operation may result in a tightening of a continuous or discrete variable bound or constraint that in turn may result in further tightenings in the subsequent iterations. The iterative procedure terminates immediately if any of the presolve operations detects an infeasibility. Alternatively, our termination condition on Line~2 of Alg.~\ref{alg:Sparse_presolve} is based on whether a particular measure of progress is sufficient or not. For example, progress can be measured by the number of variables that are fixed and/or the amount by which continuous or discrete variable bounds are tightened from one iteration to a next. The presolve procedure continues as long as one or multiple variable bounds are tightened sufficiently from one iteration to the next, i.e., if \texttt{sufficient\_progress} is true on Line~2 of Alg.~\ref{alg:Sparse_presolve}. However, after a minimum number of iterations, the algorithm may terminate if no new variable is fixed in the latest iteration, to avoid performing an excessive number of iterations for incremental tightening of the continuous variable bounds.
Due to the relatively high computational cost of binary variable probing in Alg.~\ref{alg:Probing}, probing is performed only when the amount of progress that is made by other presolve operations is not sufficient and a maximum number of iterations has not been reached, see Line~6 of Alg.~\ref{alg:Sparse_presolve}.
The overall goal of the presolve procedure is that the total time spent for removing variables and constraints and strengthening of bounds and coefficients, is considerably smaller than the reduction in B\&B solution time that is achieved by applying the presolve in each of the nodes.

\begin{algorithm}[h]
	\caption{Block-sparse presolve procedure for optimal control structured MIQP}
	\label{alg:Sparse_presolve}
	\begin{algorithmic}[1]
		\Require Optimal control structured MIQP subproblem of the form in~\eqref{eq:OCP-MIQP}, and \texttt{probing} $\in$ \{\texttt{true},\,\texttt{false}\}.
		\State $n \gets 0$.
		\While{\texttt{sufficient\_progress} $\; \land \;$ !\texttt{infeasible} $\; \land \;$ $n < n_{\mathrm{max}}$ }
		\State Update bound values $[\ubar{z}_{i},\bar{z}_{i}], \; i \in \{0,\ldots,N\}$ using forward-backward sweep in Alg.~\ref{alg:Sparse_DP_sweep}.
		\State Update bound values $[\ubar{z}_{i},\bar{z}_{i}]$ and constraint values $[\blo_{i},\bup_{i}], \; i \in \{0,\ldots,N\}$ using Alg.~\ref{alg:Process_Inequalities}.
		\State Update constraint bound values $[\blo_{i},\bup_{i}]$ and matrices $D_i, \; i \in \{0,\ldots,N\}$ using Alg.~\ref{alg:Coefficient_Strengthening}.
		\If{!\texttt{infeasible} $\; \land \;$ \texttt{probing} $\; \land \;$ !\texttt{sufficient\_progress} $\; \land \;$ $l < l_{\mathrm{max}}$}
		\State Update bound values $[\ubar{z}_{i},\bar{z}_{i}], \; i \in \{0,\ldots,N\}$ using binary variable probing in Alg.~\ref{alg:Probing}.
		\State \texttt{probing} $\gets$ \texttt{false}. 
		\EndIf
		\State $n \gets n + 1$.
		\EndWhile
		\Ensure \change{Updated MIQP data with tightened constraints} or subproblem is detected to be infeasible.
	\end{algorithmic}
\end{algorithm}


\begin{Remark}
	Other presolving techniques such as cut generation are used in general-purpose MIQP solvers, e.g., see~\cite{Achterberg2013}. However, in the present paper, we refrain from using cut generation techniques for two reasons. First, we assume that global cuts can be generated offline and included in the \change{MIQP} formulation~\eqref{eq:OCP-MIQP}. Second, standard cutting plane techniques will usually produce inequalities that couple variables across stages, which is undesirable for our block-sparsity exploiting implementation.
\end{Remark}

\section{Heuristic Presolve Method for Mixed-Integer Optimal Control} \label{sec:HEURISTICS}


Exact presolve techniques tighten constraints or prune a subset of decision variables by fixing them to the optimal values, e.g., using domain propagation, coefficient strengthening, and probing as discussed in Section~\ref{sec:PRESOLVE}.
For MIQPs, the development of effective presolve methods in commercial solvers has proven crucial in accelerating computation times~\cite{AchterbergBixbyEtAl2019}.
In this section, we present a \emph{heuristic presolve} method, which was originally proposed in~\cite{Cauligi2022} to increase rate of feasibility for supervised learning, but the approach is \change{generalized} here as a heuristic for \change{quickly} computing a feasible but suboptimal MIQP solution. The heuristic can be used to compute an \change{initial solution guess, i.e., an initial upper bound value that may result in additional pruning of nodes and therefore speedup the B\&B algorithm (e.g., see Figure~\ref{fig:BB_ilu}) or to implement a fast but suboptimal MI-MPC method}.

\subsection{Abstract Definition of a Single Presolve Step}
We refer to the parametric MIQP in~\eqref{eq:OCP-MIQP} as $\P(\theta)$, in which the problem parameters \change{$\theta$} can include the current state estimate $\hat{x}_0$, and we use $\bin \in \Z^{N_\bin}$ to denote the discrete optimization variables in~\eqref{OCP:int}. In addition, we use the compact notation $\P(\theta,\bin_{\fixed} = \hat{\bin})$ in the following abstract definition of a presolve step, to denote the MIQP~\eqref{eq:MIQP} after fixing $\bin_i = \hat{\bin}_i, i\in\fixed$ for an index set $\fixed$.

\begin{Definition}[Single Presolve Step]
	Given the problem $\P(\theta)$ and a set of integer values $\{\hat{\bin}_i\}_{i \in \fixed}$ for the index set $\fixed \subseteq \{ 1, \dots, N_\bin \}$, the presolve step computes
	\begin{equation}
	\{\flag, \hat{\bin}^+,\fixed^+\} \gets \text{Presolve}(\P(\theta),\hat{\bin},\fixed),
	\end{equation}
	resulting in updated integer values $\{\hat{\bin}_i^+\}_{i \in \fixed^+}$ for the index set $\fixed^+ \subseteq \{ 1, \dots, N_\bin \}$, for which following conditions are satisfied:
	\begin{enumerate}
		\itemsep0pt
		\item The new set of indices includes at least the original set, i.e., $\fixed \subseteq \fixed^+$.
		\item Problem $\P(\theta,\bin_{\fixed^+} = \hat{\bin}^+)$ is infeasible or unbounded, i.e., $\flag = \False$, only if $\P(\theta,\bin_{\fixed} = \hat{\bin})$ is infeasible or unbounded.
		\item Any feasible or optimal solution of $\P(\theta,\bin_{\fixed^+} = \hat{\bin}^+)$ can be mapped to a feasible or optimal solution of $\P(\theta,\bin_{\fixed} = \hat{\bin})$, and their objective values are identical.
	\end{enumerate}
	\label{def:presolve}
\end{Definition}
A presolve routine applied to a root node in B\&B corresponds to Definition~\ref{def:presolve} with $\fixed=\emptyset$.
In general, presolve cannot prune all of the binary or integer decision variables, but it can often lead to a reduced problem that is significantly faster to solve. For example, the presolve step can be defined by the block-sparse implementation in Alg.~\ref{alg:Sparse_presolve}.

\subsection{Iterative Procedure for Heuristic Presolve Method}

One heuristic approach to compute a feasible but possibly suboptimal MIQP solution is to construct a solution guess $\hat{\bin}$ for all discrete variables $\bin$ in~\eqref{eq:OCP-MIQP}, followed by solving the convex QP that results from fixing each variable $\bin_i = \hat{\bin}_i$ for $i = 1, \dots, N_\bin$. The solution guess $\hat{\bin}$ can be constructed in multiple ways, for example, by shifting the optimal MIQP solution from the previous to the next time step or by rounding a solution of the convex QP relaxation at the root node to the nearest integer solution. Alternatively, supervised learning can be used to train a recurrent neural network architecture~\cite{Cauligi2022} to make accurate predictions of the optimal values for the discrete variables in the MIQP problem $\P(\theta)$. Given such a discrete solution guess $\hat{\bin}$, we present a heuristic presolve method that can be used to compute an updated solution guess $\hat{\bin}_{\PS} $ with increased likelihood of finding a feasible and close to optimal MIQP solution by solving the resulting convex QP.

As illustrated in Figure~\ref{fig:presolve}, we propose to use a heuristic presolve method that aims to correct values in the discrete solution guess $\hat{\bin}$, given that the presolve operation can provide additional fixings that preserve feasibility and optimality; see Definition~\ref{def:presolve}. Our key insight is that if presolve prunes a subset of the integer variables that can be fixed to the optimal values, then the discrete solution guess $\hat{\bin}$ only needs to be used to fix any remaining integer decision variables, reducing the risk of causing infeasibility. The proposed method can be used to improve any existing heuristic, e.g., when presolve prunes an integer variable $\bin_i$ for which the heuristic made an incorrect prediction that would have resulted in an infeasible solution. 
Algorithm~\ref{alg:ITERATIVE_PRESOLVE} describes the proposed heuristic presolve method given a candidate integer solution $\hat{\bin}$ for a particular vector of problem parameters $\theta$.
Let $\fixed$ be the index set of pruned decision variables (i.e., $\bin_i$ for $i \in \fixed$ are fixed), $| \fixed |$ is the number of pruned integer decision variables, and $\fixed$, $\hat{\bin}_{\PS}$ are initialized as empty sets (Line~\ref{line:init_presolve_set}).
The presolve step is then called (Line~\ref{line:presolve_routine}) and returns whether infeasibility was detected, an updated set of indices $\fixed$, and corresponding values $\hat{\bin}_{\PS}$.
Algorithm~\ref{alg:ITERATIVE_PRESOLVE} terminates if all of the integer variables have been fixed (Line~\ref{line:check_cardinality}) or if infeasibility is detected (Line~\ref{line:presolve_failure}).
Otherwise, one of the remaining free integer variables $\bin_j$ for $j \in \{ 1, \dots, N_\bin \} \setminus \fixed$ is selected and fixed to the value in the original solution guess, i.e., $\hat{\bin}_{\PS,j} = \hat{\bin}_j$ (Line~\ref{line:var_select}-\ref{line:update_set}).

\begin{figure}[!ht]
	\centering
	\includegraphics[width=0.6\columnwidth,clip]{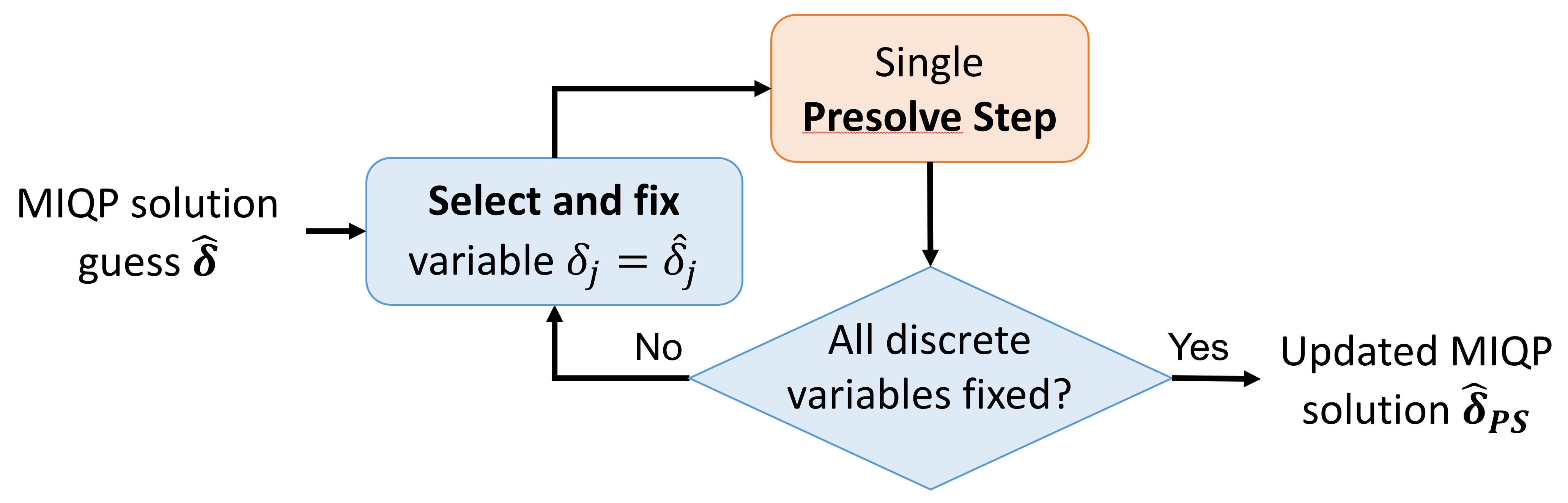}
	\caption{Heuristic presolve method to fix discrete optimization variables in \change{MIQP}.}
	\label{fig:presolve}
\end{figure}

In Algorithm~\ref{alg:ITERATIVE_PRESOLVE}, we denote $\mathtt{VarSelect}$ as the procedure to choose the integer variable $\bin_j$ to fix next, which is closely related to the variable selection policy in B\&B routines. We will further illustrate the performance of the heuristic presolve method in Alg.~\ref{alg:ITERATIVE_PRESOLVE} in combination with rounding of the relaxed QP solution at the root node to the nearest integer solution. In this case, the variable selection strategy $\mathtt{VarSelect}$ can select a remaining free integer variable, for which the relaxed QP solution is closest to integer feasible, i.e., the variable for which the difference with the value after rounding is smallest.
Algorithm~\ref{alg:ITERATIVE_PRESOLVE} is an iterative procedure, which calls our tailored presolve routine in Alg.~\ref{alg:Sparse_presolve} at each iteration~(Line~\ref{line:presolve_routine}). At an increased computational cost, one can use the binary variable probing strategy in Alg.~\ref{alg:Probing} specifically for the selected variable $\delta_j$, see Line~\ref{line:probing} of Alg.~\ref{alg:ITERATIVE_PRESOLVE}, in order to check whether the variable can be fixed based on probing before fixing it to the value in the solution guess. Alternatively, one could skip the presolve call when a particular limit on the computation time has been reached, i.e., such that all remaining free integer variables are fixed directly to \change{their corresponding values} in the solution guess (Line~\ref{line:update_set}).

\begin{algorithm}[h]
	\caption{Iterative Procedure for Heuristic Presolve Method}
	\label{alg:ITERATIVE_PRESOLVE}
	\begin{algorithmic}[1]
		\Require {Candidate integer solution} $\hat{\bin}${, problem parameters }$\theta$.
		\State Initialize set of pruned integer variables $\fixed \leftarrow \emptyset$ and values $\hat{\bin}_{\PS} \gets \emptyset$. \label{line:init_presolve_set}
		\For{$i \in \{1, \dots, N_\bin\}$}
		\State $\{\flag,\hat{\bin}_{\PS},\fixed\} \gets \text{Presolve}(\P(\theta),\hat{\bin}_{\PS},\fixed)$ in Alg.~\ref{alg:Sparse_presolve}. \label{line:presolve_routine}
		\IfThen{$\flag == \False$}{\Return $(\False,\hat{\bin}_{\PS})$.} \label{line:presolve_failure}
		\IfThen{$|\fixed| == N_\bin$}{\Return $(\True,\hat{\bin}_{\PS})$.} \label{line:check_cardinality}
		\State $j \gets $ $\mathtt{VarSelect}(\{ 1, \dots, N_\bin \} \setminus \fixed)$. \label{line:var_select}
		\State Optional: perform binary variable probing in Alg.~\ref{alg:Probing} specifically for $\delta_j$. \label{line:probing}
		\IfThen{$\delta_j$ not fixed by probing}{fix integer variable $\hat{\bin}_{\PS,j} = \hat{\bin}_{j}$, $\fixed \gets \fixed \cup \{ j \}$. \label{line:update_set}}
		\EndFor
		\Ensure{Updated integer solution guess $\hat{\bin}_{\PS}$.}
	\end{algorithmic}
\end{algorithm}


%
%
%

%
%
%
%

\section{Practical Considerations for Embedded Mixed-Integer MPC}
\label{sec:MIMPC}

In embedded real-time applications of mixed-integer MPC~(MI-MPC), one needs to solve an MIQP~\eqref{eq:OCP-MIQP} at each sampling instant under strict timing constraints. We can leverage the fact that we solve a sequence of similar problems, parametrized by the initial state value $\hat{x}_{0}$, in order to warm-start the B\&B algorithm. We refer to our warm-starting procedure as \emph{tree propagation}, which was originally proposed in~\cite{Hespanhol2019}, because the main goal is to ``propagate'' the B\&B tree forward by one time step. 


\subsection{Branch-and-bound Tree Propagation for Warm Starting}



The warm-starting procedure aims to use knowledge of one MIQP solution, i.e., the search tree after solving the problem, in order to improve the B\&B search for the next MIQP~\cite{Bemporad2018,Marcucci2021}. Our approach is to store the path from the root to the leaf node where the optimal solution to the MIQP was found, as well as the order in which the variables are branched. This is based on the knowledge that the branching order is crucial for the efficiency of any B\&B method~\cite{achterberg2005branching}. We then perform a shifting and update of this path in order to obtain a \emph{warm-started tree} to start our B\&B search at the next time step. We illustrate this procedure in Figure~\ref{fig:MPCTree}, where the optimal path at the current time step is denoted by the sequence of nodes $P_0 \rightarrow P_1 \rightarrow P_3 \rightarrow P_5 \rightarrow P_6$. Let us consider a corresponding sequence of variables $u_2 \rightarrow u_3 \rightarrow u_0 \rightarrow u_1$ that we branched on in order to find the optimal solution at the leaf node $P_6$. After shifting by one time step, all branched variables in the first control interval can be ignored. Assuming the index refers to the time step in the control horizon, this results in a shifted and shorter path of variables $u_1 \rightarrow u_2 \rightarrow u_0$.

After obtaining a new state estimate $\hat{x}_0$, we execute the block-sparse presolve routine in Alg.~\ref{alg:Sparse_presolve} and we solve the convex QP relaxation corresponding to the root node. Given the warm-started tree and the relaxed solution at the root node, we remove nodes that correspond to branched variables that are already integer feasible in the solution at the root node. The shifted MIQP solution from the previous time step can be used to fix all integer variables at the current time step and solve the resulting convex QP, which may result in a feasible solution and therefore an upper bound for the B\&B method. Alternatively or in addition, the heuristic presolve method in Alg.~\ref{alg:ITERATIVE_PRESOLVE} can be used to compute a feasible solution and possibly a tighter B\&B upper bound. We proceed by solving all the leaf nodes on the warm-started path. As we solve both children of a node on this path, we do not have to solve the parent node itself and therefore we reduce computations by solving fewer QP relaxations.
We process the warm-started tree in the order depicted by the index of each node in Fig.~\ref{fig:MPCTree}, after which we resume normal procedure of the B\&B method.

\begin{figure} 
\centering
\includegraphics[width=0.75\textwidth]{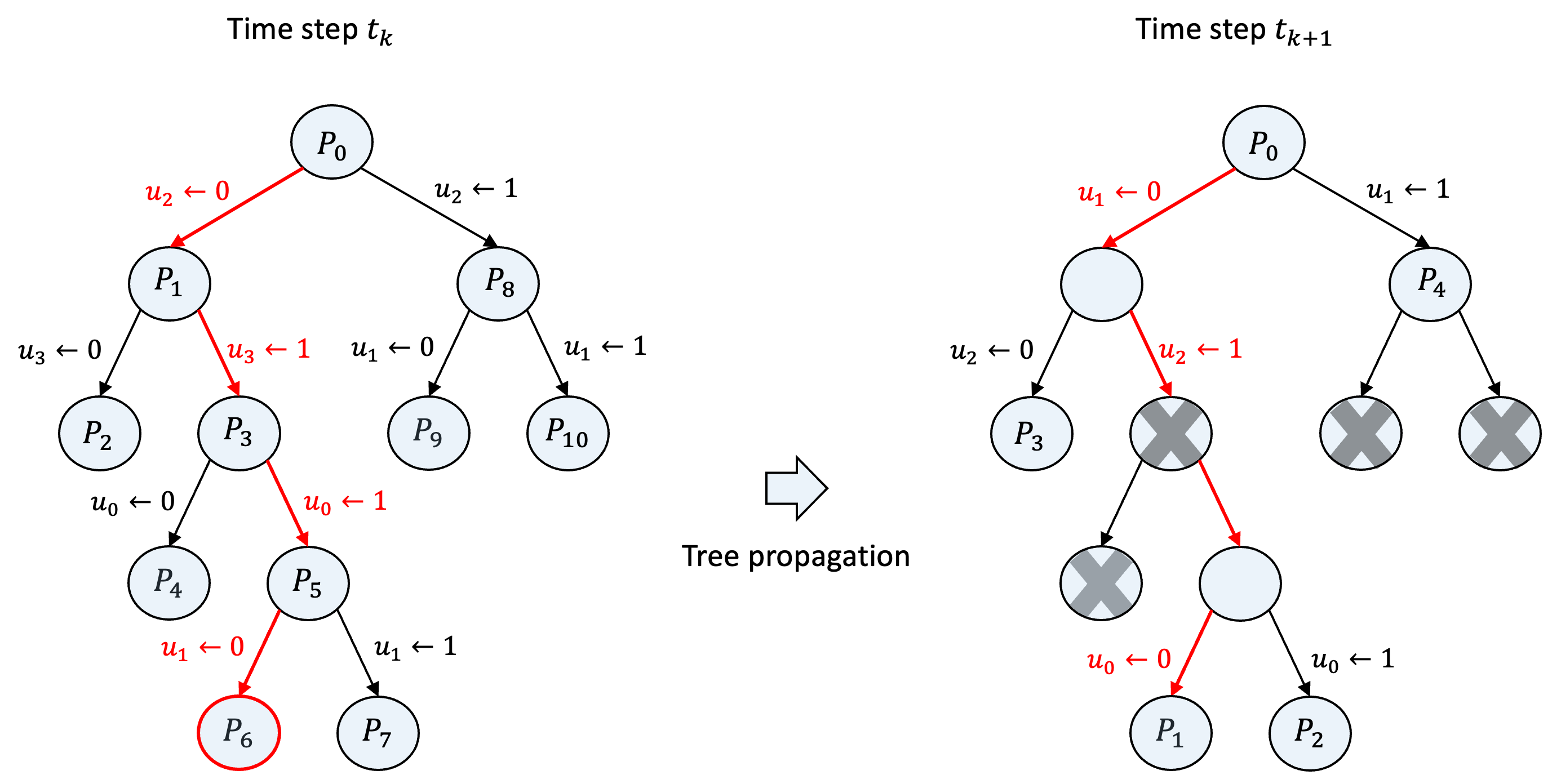}
\caption{Illustration of the tree propagation technique \change{for warm starting of the B\&B method} from one time step to the next in the MI-MPC algorithm: index~$i$ denotes the order in which each node $P_i$ is solved.}
\label{fig:MPCTree}
\end{figure}


We can additionally shift and re-use the pseudo-cost information from one MPC time step to the next, in order to make better branching decisions without the need for computationally expensive strong branching.
The propagation of pseudo-costs can be coupled with an update of the reliability parameters in order to discount relatively old pseudo-cost information, which may help to avoid increasing the worst-case computation time by making a bad branching decision. The reliability number may be reduced for each variable from one time step to the next, in order to force strong branching for variables that have not been branched on in a sufficiently long time. In addition, nodes can be removed from the warm-started path in case they correspond to branched variables for which there is no pseudo-cost information or it is not sufficiently reliable, in an attempt to avoid bad branching decisions in the B\&B method. Finally, these warm-started pseudo-costs can also be used to re-order the warm-started tree to improve the branching order and hopefully result in smaller B\&B search tree sizes.

\subsection{Embedded Software Implementation for MI-MPC}

We refer to the proposed MI-MPC algorithm as \software{BB-ASIPM}, since it combines a B\&B method with tailored block-sparse presolve techniques and \software{ASIPM} to solve the convex QP relaxations. At each control time step, it solves an  MIQP where the B\&B tree and the pseudo-cost information are warm-started. The B\&B strategy and the tailored block-sparse presolve methods, the warm-start and heuristic branching techniques, as well as the \software{ASIPM} solver have all been implemented in self-contained C~code, which allows for real-time implementations on embedded control hardware.
In addition to the presented techniques to speed up a B\&B method for MI-MPC, multiple heuristic techniques can be used to achieve real-time feasibility. For example, an upper bound on the number of B\&B iterations can be imposed to ensure a maximum computation time and to allow the MI-MPC controller to use a feasible but suboptimal solution instead. If an integer-feasible solution has been found, e.g., based on warm-starting in Fig.~\ref{fig:MPCTree} or the heuristic presolve method in Alg.~\ref{alg:ITERATIVE_PRESOLVE}, a B\&B method can provide a bound on its suboptimality. Alternatively or in addition, an integer horizon can be introduced, i.e., integer feasibility could be enforced on the first $M < N$ stages in~\eqref{OCP:int}. The latter can drastically reduce computation times at the cost of an approximation error as discussed in~\cite[Sec.~4.2]{frick2015embedded}. However, the investigation of such additional heuristics is beyond the scope of the present paper.

\section{Case Studies: Mixed-Integer MPC Simulation Results}
\label{sec:caseStudies}

Let us illustrate the computational performance of \software{BB-ASIPM} against multiple state-of-the-art MIQP solvers for two numerical case studies of mixed-integer optimal control applications. In addition, we will illustrate the performance of the heuristic presolve method to compute a feasible but possibly suboptimal solution in real time. The two benchmark problems include: (1)~mobile robot motion planning with obstacle avoidance constraints and (2)~an underactuated cart-pole with soft contacts.
All computation times in Section~\ref{sec:planning} are obtained on a MacBook Pro (16-inch, 2019), with $2.4$~GHz 8-Core Intel Core i9 processor. The numerical results in Section~\ref{sec:cartpole} are obtained using hardware-in-the-loop~(HIL) simulations on a dSPACE Scalexio rapid prototyping unit, with a DS6001 Processor Board for real-time processing~\cite{SCALEXIO2022}.

\subsection{Mixed-integer Optimal Control for Motion Planning with Obstacle Avoidance}
\label{sec:planning}

The first case study concerns the time-optimal motion planning, using a simple kinematic model for a mobile robot, taking into account collision avoidance constraints based on an MIQP formulation that is similar to the benchmark example from~\cite{Cauligi2022}.


\subsubsection{Problem Formulation}

We formulate an MIOCP of the form in~\eqref{eq:OCP-MIQP}, where the state vector is defined as $x = [\px \; \py \; \vx \; \vy \; \goal]^\top$, i.e., $\nx = 5$, including the 2D position $(\px,\py)$ and velocity vector $(\vx,\vy)$ of the robot and including a binary state $\goal \in \{0,1\}$ that indicates whether the goal has been reached. The discrete-time state dynamics in~\eqref{OCP:dyn} read as
\begin{equation}
x_{i+1} =  \begin{bmatrix} \px_{i+1} \\ \py_{i+1} \\ \vx_{i+1} \\ \vy_{i+1} \\ \goal_{i+1} \end{bmatrix} = \begin{bmatrix} 1 & 0 & \Ts & 0 & 0 \\ 0& 1 & 0 & \Ts & 0 \\ 0 & 0 & 1 & 0 & 0 \\ 0& 0 & 0 & 1 & 0 \\ 0 & 0 & 0 & 0 & 1 \end{bmatrix} \begin{bmatrix} \px_i \\ \py_i \\ \vx_i \\ \vy_i \\ \goal_i \end{bmatrix} + \begin{bmatrix} 0 & 0 & 0 \\ 0 & 0 & 0 \\ \Ts & 0 & 0 \\ 0 & \Ts & 0 \\ 0 & 0 & 1 \end{bmatrix} \begin{bmatrix} \ax_i \\ \ay_i \\ \deltag_i \end{bmatrix}, \label{eq:dyn_case1}
\end{equation}
where the control inputs are $u = [\ax \; \ay \; \deltag]^\top$, i.e., $\nU=3$, including the 2D acceleration vector $(\ax,\ay)$ and an auxiliary binary variable $\deltag_i \in \{0,1\}$ that determines whether the goal is reached at the time step $i \in \{1,\ldots,N\}$ in the prediction horizon. The inequality constraints~\eqref{OCP:ineq} include requirements on reaching the goal state as
\begin{equation}
\xgoal - \M\,(1-\goal_i) \le x_{i,1:4} \le \xgoal + \M\,(1-\goal_i), \label{eq:goal}
\end{equation}
where $\M > 0$ is sufficiently large and $x_{i,1:4} = [\px_i \; \py_i \; \vx_i \; \vy_i]^\top$, which ensures the implication $\goal_i = 1 \implies x_{i,1:4} = \xgoal$ and we additionally enforce that $\sum_{i=0}^{N-1} \deltag_i = 1$, and therefore $\goal_N=1$ holds. As illustrated in Figure~\ref{fig:motion_planning_trajs}, we include collision avoidance constraints for rectangular obstacle shapes that are aligned with the axes for simplicity, 
\begin{equation}
\begin{aligned}
\pxmin_j - \M \delta^{\mathrm{o},1}_j \le \px_i &\le \pxmin_j + \M (1 - \delta^{\mathrm{o},1}_j), \\
\pxmax_j - \M (1- \delta^{\mathrm{o},2}_j) \le \px_i &\le \pxmax_j + \M \delta^{\mathrm{o},2}_j, \\
\pymax_j - \M (1-\delta^{\mathrm{o},4}_j) \le \py_i &\le \pymin_j + \M (1-\delta^{\mathrm{o},3}_j), \\
\pxmin_j - \M (1-\delta^{\mathrm{o},3}_j-\delta^{\mathrm{o},4}_j) \le \px_i &\le \pxmax_j + \M (1-\delta^{\mathrm{o},3}_j-\delta^{\mathrm{o},4}_j), \\
\delta^{\mathrm{o},1}_j + \delta^{\mathrm{o},2}_j + \delta^{\mathrm{o},3}_j + \delta^{\mathrm{o},4}_j &= 1,
\end{aligned} \label{eq:obstacle_constraints}
\end{equation}
where the bounds $[\pxmin_j \; \pxmax_j]$ and $[\pymin_j \; \pymax_j]$ define the rectangular avoidance zone for each obstacle $j = 1, \ldots, n_{\mathrm{obs}}$. As can be observed in Figure~\ref{fig:motion_planning_trajs}, the obstacle avoidance constraints in~\eqref{eq:obstacle_constraints} are enforced pointwise in time, and each exclusion zone includes appropriate margins to account for the physical dimension of the mobile robot and to account for discretization errors. The cost function~\ref{OCP:obj} for this numerical case study reads as
\begin{equation}
\sum_{i=0}^{N-1} \left( i\, \deltag_i + \Vert u_{i,1:2} \Vert_R^2 \right) + \sum_{i=0}^{N} \Vert x_{i,1:4} - \xgoal \Vert_Q^2,
 \label{eq:cost_case1}
\end{equation}
where the first term accounts for minimizing the time step at which the goal is reached and minimizing the acceleration inputs $u_{i,1:2} = [\ax_i \; \ay_i]^\top$, and the second term accounts for minimizing the state error with respect to the goal state.

\begin{figure}[h] 
	\centerline{\hbox{
			\includegraphics[width=0.75\textwidth]{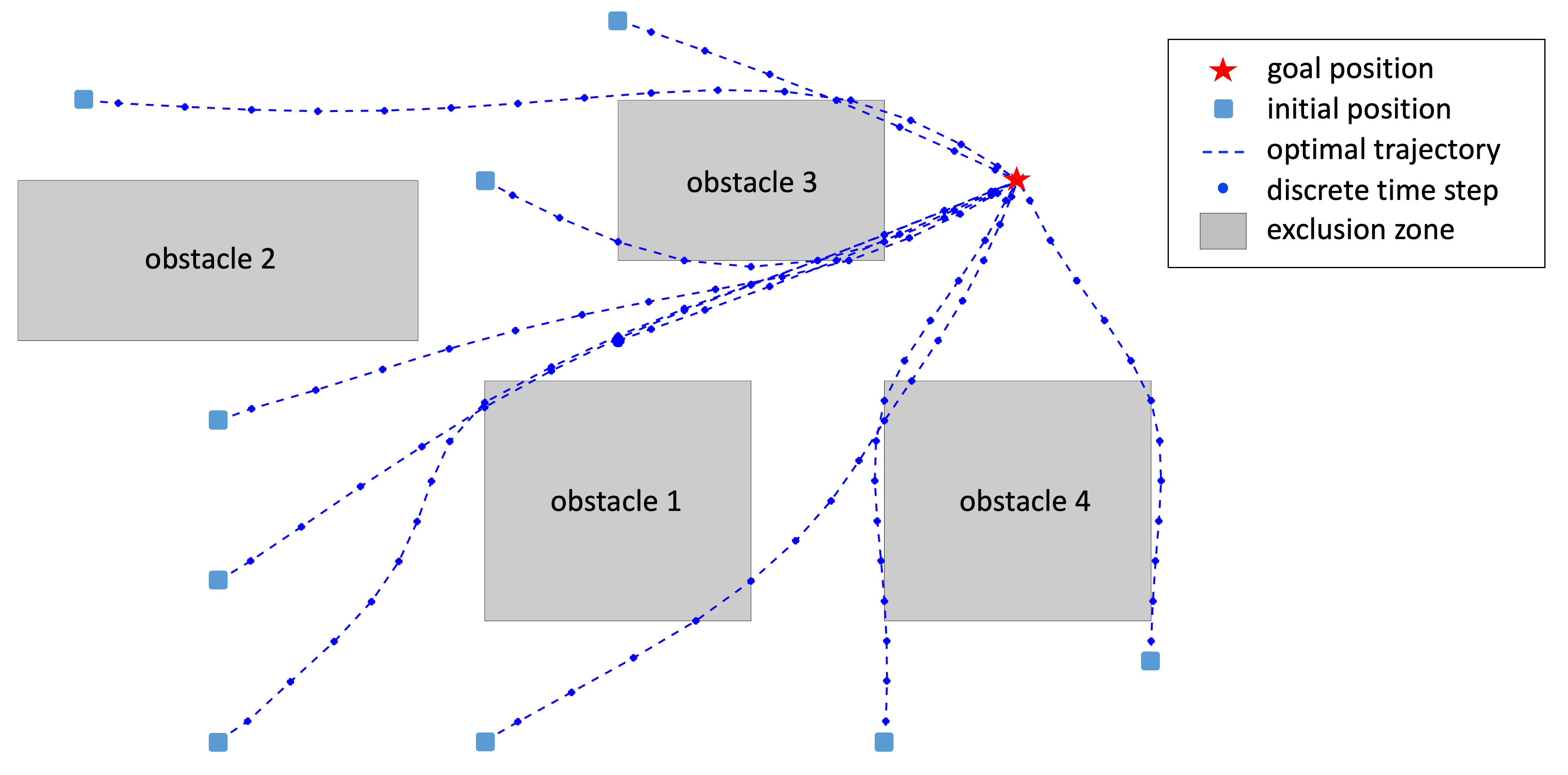}}}
	\caption{Illustration of the MIP-based motion planning case study: the optimal trajectories for a horizon length $N = 20$ are depicted in blue, given four rectangular obstacle shapes, multiple initial positions and a fixed goal position. \change{The obstacle avoidance constraints are enforced pointwise in time, i.e., for each discrete time step of the trajectories, and the obstacle dimensions include appropriate margins to account for discretization errors and to account for the physical shape of the mobile robot.}}
	\label{fig:motion_planning_trajs}
\end{figure}

\subsubsection{MIQP Solver: Numerical Results}

We illustrate the computational performance of the proposed MIQP solver by solving multiple MIOCPs for motion planning with obstacle avoidance for a varying number of obstacles $n_{\mathrm{obs}}$, a varying horizon length $N$ and varying initial state values $\hat{x}_0$. In particular, the number of obstacles is selected to be in the range $n_{\mathrm{obs}} \in [1,4]$ and the horizon length $N \in [6,20]$. The MIOCP for each horizon length and for each number of obstacles is solved for $200$ different initial state values, sampled from a uniform distribution where $0 \le \px_0 \le 10$ and $0 \le \py_0 \le 10$, excluding the avoidance regions. In addition to our proposed \software{BB-ASIPM} solver, the MIOCPs are solved using the state-of-the-art solvers \software{GUROBI}~\cite{gurobi}, \software{MOSEK}~\cite{mosek}, \software{GLPK}~\cite{GLPK2022}, \software{Cbc}~\cite{CBC2022}, and Matlab's \software{intlinprog}. Because the solvers \software{GLPK}, \software{Cbc}, and \software{intlinprog} do not provide direct support for the solution of MIQPs, we present numerical results for solving a mixed-integer linear programming~(MILP) approximation with these particular software tools instead.
Figure~\ref{fig:comparison_mp} shows the average computation time of the resulting MIOCP solutions for $200$ randomly sampled initial state values, for each of the numbers of obstacles $n_{\mathrm{obs}} \in [1,4]$ and for each horizon length $N \in [6,20]$. As expected, the average computation time for motion planning grows rapidly with the number of obstacles $n_{\mathrm{obs}}$ and with the horizon length $N$, due to the increasing number of binary optimization variables $n_{\mathrm{bin}} = N\, \left( 1 + 4\, n_{\mathrm{obs}} \right)$ in the considered MIOCP formulation.

\begin{figure}[h] 
	\centerline{\hbox{
			\includegraphics[width=1.2\textwidth]{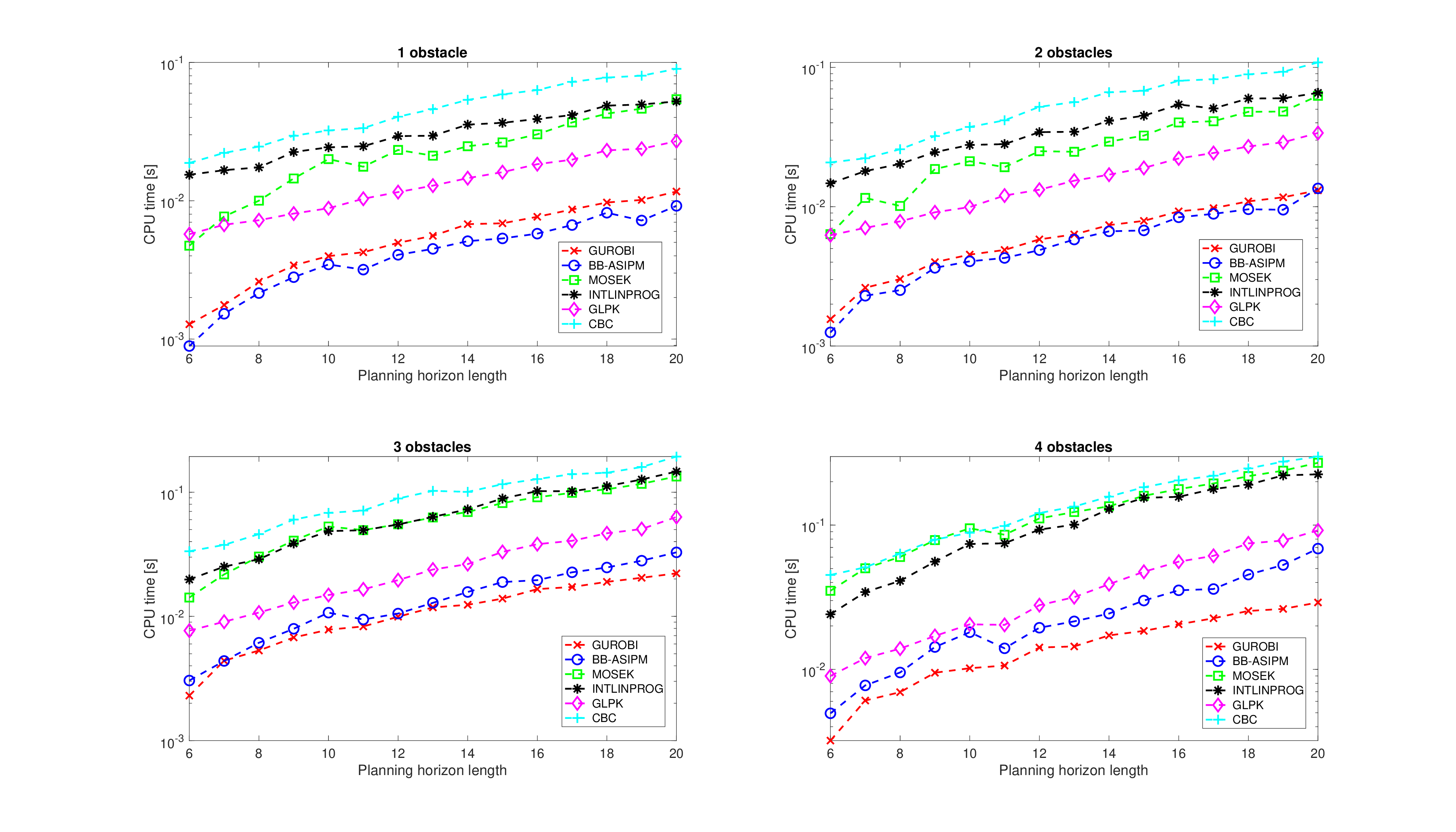}}}
	\caption[]{Average computation times for MIP-based motion planning case study for $1$, $2$, $3$ and $4$ obstacles, using a varying horizon length $N \in [6,20]$, and using solvers \software{GUROBI}, \software{MOSEK}, \software{intlinprog}, \software{GLPK}, \software{Cbc} and the proposed \software{BB-ASIPM} solver. Note that each of the solvers computes the globally optimal solution of the MIP.}
\label{fig:comparison_mp}
\end{figure}

For this particular case study, the numerical results in Figure~\ref{fig:comparison_mp} show that our proposed \software{BB-ASIPM} is the fastest solver for all horizon lengths $N \in [6,20]$, when $n_{\mathrm{obs}}=1$ and $n_{\mathrm{obs}}=2$. The state-of-the-art commercial solver \software{GUROBI} becomes faster for $n_{\mathrm{obs}} \ge 3$ obstacles and it will likely continue to outperform the \software{BB-ASIPM} solver for larger problem dimensions, because \software{GUROBI} 
includes many advanced cutting plane techniques and heuristics~\cite{gurobi} that are not implemented in the \software{BB-ASIPM} solver. However, \software{BB-ASIPM} remains relatively competitive and it outperforms the other state-of-the-art solvers even for larger MIOCPs with up to $n_{\mathrm{bin}} = 340$ binary optimization variables for $N=20$ and $n_{\mathrm{obs}}=4$, in this particular case study. These results confirm that our proposed \software{BB-ASIPM} solver can be competitive with state-of-the-art MIP solvers, even though the software implementation of \software{BB-ASIPM} is relatively compact and self-contained such that it can be executed on an embedded microprocessor for real-time applications of mixed-integer optimal control, e.g., for real-time motion planning with obstacle avoidance constraints. Instead, state-of-the-art optimization tools, such as \software{GUROBI} typically cannot be used on embedded control hardware with limited computational resources and with limited memory~\cite{DiCairano2018tutorial}. Some results of using \software{BB-ASIPM} on the dSPACE Scalexio rapid prototyping unit will be illustrated further in Section~\ref{sec:cartpole}.

\begin{Remark}
	Advanced presolve and heuristic options have been activated for each of the state-of-the-art software tools, resulting in fair computational comparisons. However, because the C~code implementation of the \software{BB-ASIPM} solver is a single-threaded process, the computation times in this paper are based on single-threaded processing for all MIP solvers. Similar to state-of-the-art B\&B methods, parallel processing could be used to speed up computations in the \software{BB-ASIPM} solver, by solving QP relaxations corresponding to multiple nodes in the B\&B tree simultaneously. This remains outside the scope of the present paper.
\end{Remark}

\subsubsection{Heuristic Presolve Method: Numerical Results}

Table~\ref{tab:iterative_presolve} shows the average computation time of using \software{GUROBI}, \software{MOSEK}, or \software{BB-ASIPM} to solve the MIOCPs for $200$ randomly sampled initial state values, for each of the numbers of obstacles $n_{\mathrm{obs}} \in \{1,3,5,7\}$ and for different horizon lengths $N \in \{6,12,18\}$. In particular, Table~\ref{tab:iterative_presolve} also reports the average computation time for the \software{Heuristic-presolve} technique in Algorithm~\ref{alg:ITERATIVE_PRESOLVE}.
The candidate binary solution $\hat{\bin}$, which is an important input that strongly affects the performance of Algorithm~\ref{alg:ITERATIVE_PRESOLVE}, is computed by rounding the relaxed QP solution at the root node to the nearest integer solution for the numerical results in Table~\ref{tab:iterative_presolve}. Since the \software{Heuristic-presolve} method cannot guarantee to find the globally optimal solution, and it may even fail to find a feasible solution of the MIQP, Table~\ref{tab:iterative_presolve} reports the rate of infeasibility and suboptimality as
\begin{equation}
\text{Infeasibility} = 100\, \frac{n_{\mathrm{fail}}}{n_{\mathrm{sample}}} [\%], \qquad \text{Suboptimality} = 100\, \frac{J - J^\star}{J^\star} [\%], \label{eq:infeas}
\end{equation}
where $n_{\mathrm{sample}}=200$, $J^\star$ is the globally optimal objective value for a particular MIQP and $J$ is the objective value of a feasible but potentially suboptimal solution to the same MIQP, i.e., $J \ge J^\star$. Similar to Figure~\ref{fig:comparison_mp}, the numerical results in Table~\ref{tab:iterative_presolve} show that \software{BB-ASIPM} can outperform \software{GUROBI} for relatively small problems, but the \software{GUROBI} solver is faster for larger problem dimensions. In addition, Table~\ref{tab:iterative_presolve} illustrates how the proposed \software{Heuristic-presolve} technique can be used to compute a suboptimal solution faster than the time needed to solve the same MIQP with either \software{GUROBI}, \software{MOSEK} or \software{BB-ASIPM}. The rate of infeasibility and suboptimality of \software{Heuristic-presolve} is small for most cases of this particular problem, but it can be observed that these rates generally increase with the MIQP problem dimensions. It has been shown recently in~\cite{Cauligi2022} that the performance of \software{Heuristic-presolve} can be improved, i.e., the rate of infeasibility and suboptimality can be decreased, by using supervised learning techniques. Specifically, the work in~\cite{Cauligi2022} describes a framework based on deep recurrent neural networks in combination with the \software{Heuristic-presolve} in Algorithm~\ref{alg:ITERATIVE_PRESOLVE}. The details of this are outside the scope of the present paper.



\begin{table}[h]
	\caption{Average computation times for MIP-based motion planning case study for $1$, $3$, $5$ and $7$ obstacles, with a varying horizon length $N \in \{6,12,18\}$, and using solvers \software{GUROBI}, \software{MOSEK}, \software{BB-ASIPM} and the proposed \software{Heuristic-presolve} technique in Algorithm~\ref{alg:ITERATIVE_PRESOLVE}. We additionally report the rate of infeasibility and suboptimality in~\eqref{eq:infeas} for \software{Heuristic-presolve}.}
	\label{tab:iterative_presolve}
	\centering
	\setlength{\tabcolsep}{0.7em}
	\begin{tabular}{ l | c | c | c | c c c }
		\toprule
		 & \multicolumn{1}{c}{\software{GUROBI}} & \multicolumn{1}{c}{\software{MOSEK}} & \multicolumn{1}{c}{\software{BB-ASIPM}} & \multicolumn{3}{c}{\software{Heuristic-presolve}} \\
		& Time~[ms] & Time~[ms] & Time~[ms] & Time~[ms] & Infeasibility~[\%] & Suboptimality~[\%] \\
		\midrule

$N = 6, \; n_{\mathrm{obs}} = 1$  &  $1.2$~ms  &  $4.1$~ms  &  $0.8$~ms  &  $0.7$~ms  &  $0.0$~\%  &  $0.0$~\%  \\ 
$N = 6, \; n_{\mathrm{obs}} = 3$  &  $2.3$~ms  &  $13.6$~ms  &  $2.9$~ms  &  $1.4$~ms  &  $0.0$~\%  &  $0.1$~\%  \\ 
$N = 6, \; n_{\mathrm{obs}} = 5$  &  $3.2$~ms  &  $34.7$~ms  &  $5.4$~ms  &  $2.4$~ms  &  $0.0$~\%  &  $1.1$~\%  \\ 
$N = 6, \; n_{\mathrm{obs}} = 7$  &  $4.3$~ms  &  $46.8$~ms  &  $8.0$~ms  &  $3.6$~ms  &  $0.0$~\%  &  $2.4$~\%  \\ 
$N = 12, \; n_{\mathrm{obs}} = 1$  &  $4.2$~ms  &  $22.2$~ms  &  $3.8$~ms  &  $2.1$~ms  &  $0.0$~\%  &  $0.2$~\%  \\ 
$N = 12, \; n_{\mathrm{obs}} = 3$  &  $10.7$~ms  &  $75.3$~ms  &  $14.4$~ms  &  $5.7$~ms  &  $0.0$~\%  &  $0.2$~\%  \\ 
$N = 12, \; n_{\mathrm{obs}} = 5$  &  $13.6$~ms  &  $126.2$~ms  &  $23.5$~ms  &  $8.5$~ms  &  $0.0$~\%  &  $2.6$~\%  \\ 
$N = 12, \; n_{\mathrm{obs}} = 7$  &  $22.2$~ms  &  $220.9$~ms  &  $43.3$~ms  &  $15.4$~ms  &  $0.0$~\%  &  $4.8$~\%  \\ 
$N = 18, \; n_{\mathrm{obs}} = 1$  &  $7.1$~ms  &  $35.8$~ms  &  $6.4$~ms  &  $3.1$~ms  &  $0.5$~\%  &  $0.0$~\%  \\ 
$N = 18, \; n_{\mathrm{obs}} = 3$  &  $18.8$~ms  &  $140.3$~ms  &  $27.9$~ms  &  $8.5$~ms  &  $1.2$~\%  &  $0.1$~\%  \\ 
$N = 18, \; n_{\mathrm{obs}} = 5$  &  $28.4$~ms  &  $264.8$~ms  &  $55.8$~ms  &  $15.8$~ms  &  $1.3$~\%  &  $4.6$~\%  \\ 
$N = 18, \; n_{\mathrm{obs}} = 7$  &  $37.5$~ms  &  $333.6$~ms  &  $91.3$~ms  &  $26.0$~ms  &  $1.9$~\%  &  $6.5$~\%  \\ 
		\bottomrule
	\end{tabular}
\end{table}

\subsection{Mixed-integer MPC for Stabilization of Inverted Pendulum with Soft Contacts}
\label{sec:cartpole}

As a second case study, we consider MI-MPC for an inverted pendulum with soft contacts as shown in Fig.~\ref{fig:cartpole}, which is representative of underactuated, multi-contact control problems~\cite{AydinogluPreciadoEtAl2020,Todorov2012}.
Although MIOCPs form an attractive framework for modeling such problems, controllers that can react and plan online for such systems are rarely real-time feasible or near-optimal~\cite{MarcucciDeitsEtAl2017}.

\begin{figure}
	\centering
	\begin{subfigure}[b]{0.49\textwidth}
		\centerline{\hbox{
		\includegraphics[trim={0 1.5cm 0 0},clip,width=0.75\textwidth]{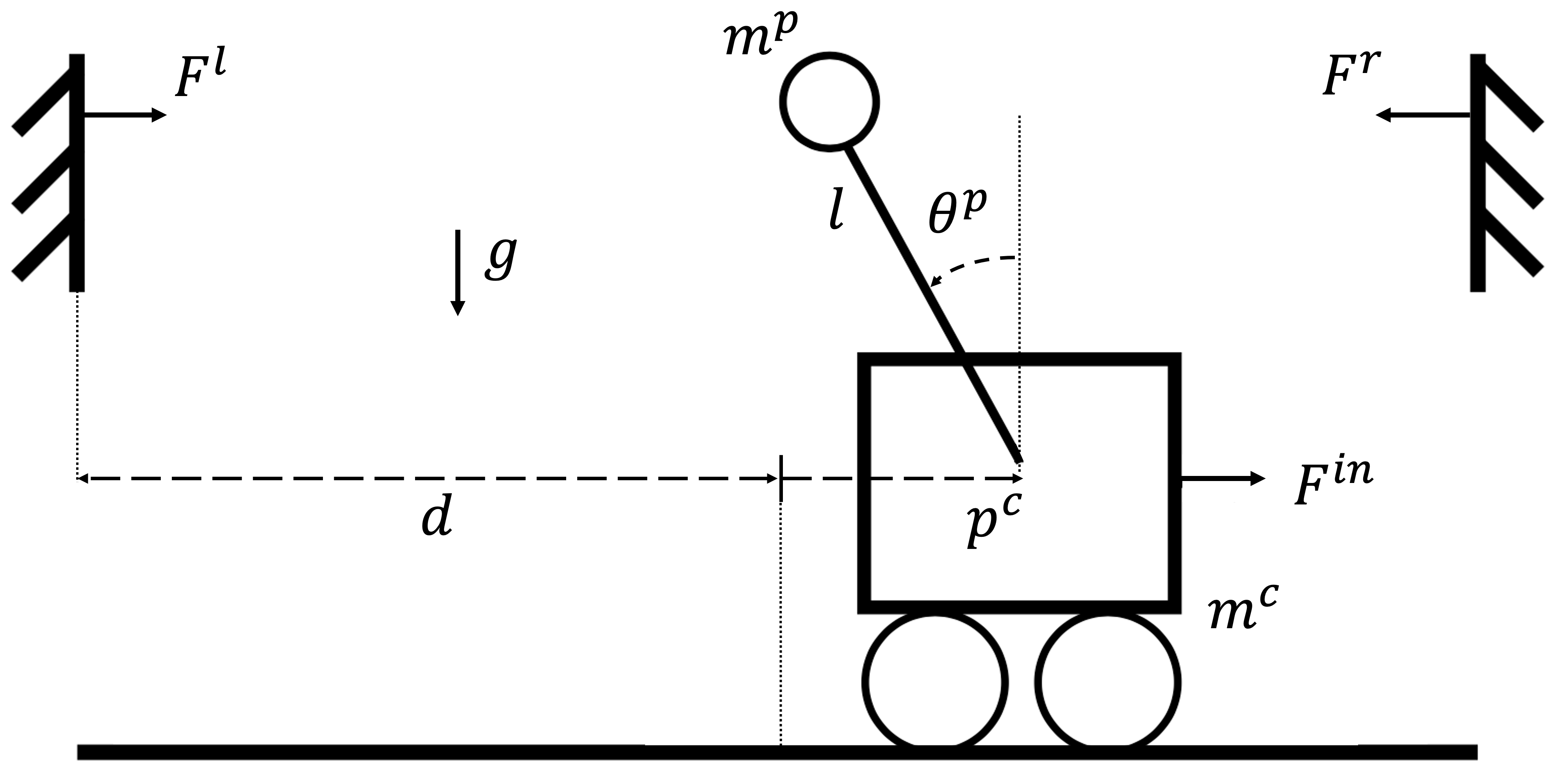}}}
\caption[]{}
\label{fig:cartpole}
\end{subfigure}
\hspace{-12mm}
\begin{subfigure}[b]{0.49\textwidth}
\centerline{\hbox{
		\includegraphics[trim={0 0 0 0},clip,width=0.5\textwidth]{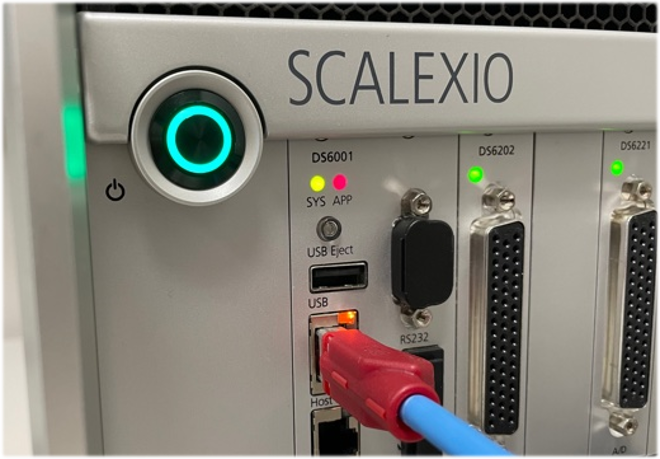}}}
\caption[]{}
\label{fig:scalexio}
\end{subfigure}
\caption{Illustration of the stabilization of an inverted pendulum on top of a cart with two soft walls that result in contact forces in~(a), and the dSPACE Scalexio rapid prototyping unit for hardware-in-the-loop simulations in~(b).}
\end{figure}

\subsubsection{Problem Formulation}

The cart-pole system with soft walls is often used as a benchmark problem for mixed-integer optimal control, e.g., see~\cite{Marcucci2021,Cauligi2022a} . This standard hybrid control problem results in an \change{MIQP} of the form in~\eqref{eq:OCP-MIQP}, which is solved at each sampling time step of the MI-MPC controller. Given the current state of the cart-pole system $\hat{x}_0$, the aim is to use the horizontal forces of the cart to effectively regulate an inverted pendulum towards the unstable equilibrium at the origin. The state vector is defined as $x = [\pc \; \thetap \; \vc \; \omegap]^\top$, i.e., $\nx = 4$, including the position $\pc$ and velocity $\vc$ of the cart, and the angle $\thetap$ and angular velocity $\omegap$ of the pendulum. The force input applied to the cart is $\Fc \in \R$ and $\Fl, \Fr \in \R$ are the contact forces transmitted by the soft walls on the left and right of the cart-pole system, see Fig.~\ref{fig:cartpole}. The discrete-time state dynamics in~\eqref{OCP:dyn} read as
\begin{equation}
x_{i+1} =  \begin{bmatrix} \pc_{i+1} \\ \thetap_{i+1} \\ \vc_{i+1} \\ \omegap_{i+1} \end{bmatrix} = \begin{bmatrix} 1 & 0 & \Ts & 0 \\ 0& 1 & 0 & \Ts \\ 0 & g \frac{\mp}{\mc} & 1 & 0 \\ 0 & g \frac{\mc+\mp}{\mc\, l} & 0 & 1 \end{bmatrix} \begin{bmatrix} \pc_i \\ \thetap_i \\ \vc_i \\ \omegap_i  \end{bmatrix} + \begin{bmatrix} 0 & 0 & 0 \\ 0 & 0 & 0 \\ \frac{1}{\mc} & 0 & 0 \\ \frac{1}{\mc\, l} & -\frac{1}{\mp\, l} & \frac{1}{\mp\, l} \end{bmatrix} \begin{bmatrix} \Fc_i \\ \Fl_i \\ \Fr_i \end{bmatrix}, \label{eq:dyn_case2}
\end{equation}
which are linearized around the nominal angle of the pendulum $\thetap=0$, and where $g$ is the gravitational acceleration, $l$ is the length of the pendulum, and $\mc$, $\mp$ are the mass of the cart and pole, respectively. We use the same parameter values as in~\cite{Marcucci2021}, i.e., $g=10$, $l=1$, and $\mc=\mp=1$.
The inequality constraints~\eqref{OCP:ineq} include the following bounds on states and inputs
\begin{equation}
-d \le \pc_i \le d, \quad -\frac{\pi}{10} \le \thetap_i \le \frac{\pi}{10}, \quad -1 \le \vc_i, \omegap_i, \Fc_i \le 1,
\label{eq:bounds_case2}
\end{equation}
where $d=0.5$ is half the distance between the two walls~(see Fig.~\ref{fig:cartpole}). A linear approximation for the position of the tip of the pole is $\pt_i = \pc_i - l\, \thetap_i$, resulting in the relative distances $\distl_i = -d - \pt_i$ and $\distr_i = \pt_i - d$ (each relative distance is positive in case of penetration), and we additionally define the time derivatives $\dotdistl_i = -\vc_i + l\, \omegap_i$ and $\dotdistr_i = \vc_i - l\, \omegap_i$ with respect to the left and right wall, respectively. 
We use the piecewise-linear definition of the contact forces
\begin{equation}
F_i^j = \left\{
\begin{array}{ll}
\kappa \dist_i^j + \nu \dot{\dist}_i^j &\quad \text{if} \;\; \dist_i^j \ge 0 \;\;\text{and}\;\; \kappa \dist_i^j + \nu \dot{\dist}_i^j \ge 0, \\
0 &\quad \text{otherwise},
\end{array}
\right. \label{eq:contact}
\end{equation}
for each of the soft walls $j\in\{\mathrm{l}, \mathrm{r}\}$, where $\kappa=100$ is the stiffness and $\nu = 10$ is the damping in the contact model. As described further in~\cite{Marcucci2021}, two binary indicator variables can be used to model each of the piecewise-linear functions in~\eqref{eq:contact}, resulting in a total of $n_{\mathrm{bin}} = 4 N$ binary optimization variables in the considered MIOCP formulation.

\subsubsection{MIQP Solver: Simulation Results on dSPACE Scalexio Hardware}

We illustrate the performance of the embedded \software{BB-ASIPM} solver using closed-loop simulations on the dSPACE Scalexio rapid prototyping unit, \change{see Fig.~\ref{fig:scalexio},} where the MI-MPC controller solves the MIOCP at each control time step with a sampling time period $\Ts^{\mathrm{mpc}}=0.1$~s and $N=8$ control intervals. The cart-pole system is simulated using the linearized dynamics in~\eqref{eq:dyn_case2} with a smaller discretization time $\Ts=0.05$~s, and the state feedback includes zero-mean Gaussian sensor noise $w_i \sim \mathcal{N}(0,\sigma_x^2)$ with covariance matrix $\sigma_x^2 = \mathrm{diag}(\sigma_{\pc}^2, \sigma_{\thetap}^2, \sigma_{\vc}^2, \sigma_{\omegap}^2)$ and standard deviation values $\sigma_{\pc} = 10^{-2}$, $\sigma_{\thetap} = 10^{-3}$, $\sigma_{\vc} = 10^{-2}$ and $\sigma_{\omegap} = 10^{-3}$. Figure~\ref{fig:cartpole_traj_N8} shows the resulting closed-loop MI-MPC trajectories for stabilization of the inverted pendulum with soft contacts, with the initial state values $\pc_0 = \change{0.25}$, $\thetap_0 = -0.05$, $\vc_0 = 0$ and $\omegap_0 = 0$, and using the \software{BB-ASIPM} solver on the dSPACE Scalexio rapid prototyping unit. For this particular set of initial state values, it can be seen from Figure~\ref{fig:cartpole_traj_N8} that the tip of the pendulum makes contact with the \change{right} wall before making contact with the \change{left} wall, after which the MI-MPC controller stabilizes the cart-pole system towards the unstable equilibrium at the origin. The behavior of the MI-MPC controller strongly depends on the cost function and the horizon length in the MIOCP but, since the cart-pole system is just an illustrative case study, a thorough tuning process remains outside the scope of the present paper. Figure~\ref{fig:cartpole_traj_N8} also shows the computation times and the number of B\&B iterations at each control time step. Finally, Table~\ref{tab:cartpole_cases} shows the average and worst-case results for the computation times, number of B\&B iterations and total number of IPM iterations to solve the convex QP relaxations in the \software{BB-ASIPM} solver, using four different sets of initial state values. It can be observed that the worst-case computation times of the \software{BB-ASIPM} solver on the dSPACE Scalexio rapid prototyping unit remain below the sampling time period of $\Ts^{\mathrm{mpc}}=0.1$~s.

\begin{figure}[h] 
	\centerline{\hbox{
			\includegraphics[trim={2cm 1.5cm 1cm 0},clip,width=0.85\textwidth]{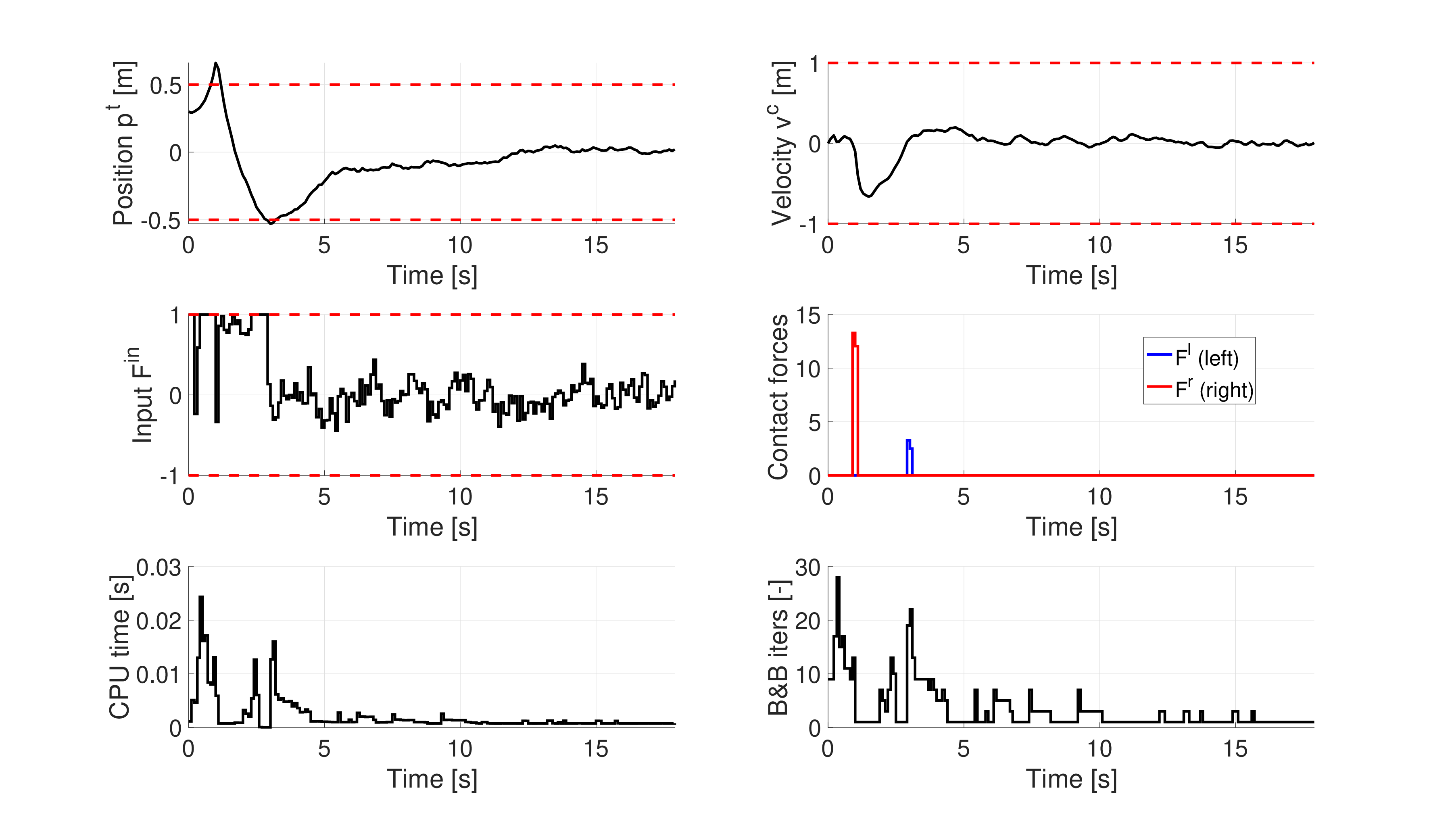}}}
\caption[]{Closed-loop MI-MPC trajectories of simulations for stabilization of the inverted pendulum with soft contacts, with initial state values $\pc_0 = \change{0.25}$ and $\thetap_0 = -0.05$, and using the \software{BB-ASIPM} solver on dSPACE Scalexio rapid prototyping unit. The red dashed lines in the upper left plot indicate the positions at which the tip of the pole penetrates the soft walls, i.e., when $\pt_i \le -d$ or $\pt_i \ge d$, \change{corresponding} to times when the contact forces in~\eqref{eq:contact} can be nonzero from the left or right wall, respectively.}
\label{fig:cartpole_traj_N8}
\end{figure}

\begin{table}[h]
	\caption{Average and worst-case computational results of the \software{BB-ASIPM} solver using four different MI-MPC simulations for stabilization of the inverted pendulum with soft contacts on the dSPACE Scalexio rapid prototyping unit.}
	\label{tab:cartpole_cases}
	\centering
	\setlength{\tabcolsep}{1.1em}
	\begin{tabular}{ l l | c c | c c | c c }
		\toprule
		 \multicolumn{2}{c}{Initial state value $\hat{x}_0$} & \multicolumn{2}{c}{CPU time~[ms]} & \multicolumn{2}{c}{B\&B iters~[-]} & \multicolumn{2}{c}{IPM iters~[-]} \\
		& & mean & max & mean & max & mean & max \\
		\midrule
$\pc_0 = 0.25$ & $\thetap_0 = -0.05$  &  \change{$2.1$~ms}  &  \change{$24.4$~ms}  &  \change{$3.2$}  &  \change{$28$}  &  \change{$39.2$}  &  \change{$569$}  \\ 
$\pc_0 = 0.00$ & $\thetap_0 = -0.05$  &  \change{$2.3$~ms}  &  \change{$53.7$~ms}  &  \change{$3.2$}  &  \change{$48$}  &  \change{$43.4$}  &  \change{$1223$}  \\ 
$\pc_0 = 0.00$ & $\thetap_0 = 0.05$  &  \change{$2.3$~ms}  &  \change{$59.5$~ms}  &  \change{$2.9$}  &  \change{$53$}  &  \change{$46.7$}  &  \change{$1406$}  \\ 
$\pc_0 = -0.25$ & $\thetap_0 = 0.05$  &  \change{$1.8$~ms}  &  \change{$27.8$~ms}  &  \change{$2.4$}  &  \change{$25$}  &  \change{$34.4$}  &  \change{$693$}  \\ 
		\bottomrule
	\end{tabular}
\end{table}

\section{Conclusions and Outlook} \label{sec:concl}

We proposed a solver for mixed-integer optimal control problems aimed at achieving a real-time implementation of mixed-integer model predictive control~(MI-MPC) on embedded platforms based on tailored presolve methods, both exact techniques and heuristics, and using a tailored solver for convex relaxations. In particular, we
provided an overview of recent work on the efficient implementation of infeasibility detection and early termination of an active-set based interior point method~(ASIPM) for convex quadratic programming. We then proposed a novel collection of block-sparse presolve techniques to efficiently remove decision variables, and to remove or tighten inequality constraints in mixed-integer quadratic programming~(MIQP).
In addition, we showed how the tailored presolve routine can be used in a novel iterative heuristic approach to compute feasible but possibly suboptimal \change{MIQP} solutions. Based on a self-contained C~code implementation of a branch-and-bound~(B\&B) method, in combination with these tailored presolve techniques and ASIPM, the resulting \software{BB-ASIPM} solver can be used to implement MI-MPC on embedded microprocessors. We presented benchmarking results for the \software{BB-ASIPM} algorithm compared against state-of-the-art MIQP solvers, including \software{GUROBI}, \software{MOSEK}, \software{GLPK}, \software{Cbc}, and Matlab's \software{intlinprog}, based on a case study of mobile robot motion planning with obstacle avoidance constraints. Finally, we demonstrated the computational performance of the \software{BB-ASIPM} solver on a dSPACE Scalexio rapid prototyping unit, using a second case study of stabilization for an underactuated cart-pole with soft contacts. 
Future work involves the validation of the \software{BB-ASIPM} solver in real-world experiments \change{using the proposed MI-MPC method for online trajectory generation, e.g., in combination with a standard MPC controller for reference tracking}.
	

\change{
\paragraph*{Data Availability Statement}
The data that support the findings of this study are available from the corresponding author upon reasonable request.
}

\bibliography{IEEEabrv,references}

\end{document}